\documentclass[11pt,twoside,leqno]{article}
\usepackage{amsmath,amsthm,amssymb}
\numberwithin{equation}{section}
\usepackage[dvips]{graphicx,color,psfrag}

\makeatletter
\newcommand{\figcaption}[1]{\def\@captype{figure}\caption{#1}}
\newcommand{\tblcaption}[1]{\def\@captype{table}\caption{#1}}
\makeatother

\newcommand{\mG}{\mathbf{\G}}

\newcommand{\baS}{\mathcal{S}}
\newcommand{\baC}{\mathcal{C}}

\newcommand{\Det}{\mathrm{Det}}

\makeatletter

\ifcase\@ptsize
\def\rpkern{\mathchoice{\kern-1.45em}{\kern-1.11em}{}{}}%
\def\grpkern{\mathchoice{\kern-1.013em}{\kern-0.825em}{}{}}%
\or
\def\rpkern{\mathchoice{\kern-1.44em}{\kern-1.11em}{}{}}%
\def\grpkern{\mathchoice{\kern-1.00em}{\kern-0.81em}{}{}}%
\or
\def\rpkern{\mathchoice{\kern-1.472em}{\kern-1.14em}{}{}}%
\def\grpkern{\mathchoice{\kern-1.00em}{\kern-0.815em}{}{}}%
\fi

\def\minibullet{\mathchoice%
{\raise0.2ex\hbox{$\scriptstyle\bullet$}}%
{\raise0.26ex\hbox{$\scriptscriptstyle\bullet$}}{}{}}
\def\butabullet{\mathchoice%
{\raise0.8ex\hbox{$\scriptstyle\bullet$}{\kern-0.365em}%
\lower0.4ex\hbox{$\scriptstyle\bullet$}}%
{\raise0.75ex\hbox{$\scriptscriptstyle\bullet$}{\kern-0.335em}%
\lower0.25ex\hbox{$\scriptscriptstyle\bullet$}}{}{}}

\def\customprod#1#2%
{\operatorname{\coprod\kern#2{#1}\kern#2\prod}\displaylimits}

\newcommand{\bA}{\mathbb{A}}

\newcommand{\bC}{\mathbb{C}}

\newcommand{\bH}{\mathbb{H}}

\newcommand{\bN}{\mathbb{N}}

\newcommand{\bQ}{\mathbb{Q}}
\newcommand{\bR}{\mathbb{R}}

\newcommand{\bZ}{\mathbb{Z}}

\renewcommand{\a}{\alpha}
\renewcommand{\b}{\beta}
\newcommand{\g}{\gamma}
\renewcommand{\d}{\delta}
\newcommand{\e}{\varepsilon}
\newcommand{\z}{\zeta}

\renewcommand{\l}{\lambda}

\newcommand{\n}{\nu}

\newcommand{\vp}{\varphi}

\newcommand{\G}{\Gamma}
\newcommand{\D}{\Delta}

\renewcommand{\Re}{\mathrm{Re}\,}
\renewcommand{\Im}{\mathrm{Im}\,}

\renewcommand{\det}{\mathrm{det}\,}

\newcommand{\Spec}{\mathrm{Spec}\,}

\newcommand{\Gauss}[1]{\lfloor{#1}\rfloor}

\newcommand{\GL}{\mathrm{GL}}

\newcommand{\Prim}{\mathrm{Prim}\,}
\newcommand{\Hyp}{\mathrm{Hyp}\,}

\newcommand{\p}{\partial}
\newcommand{\DS}[1]{\displaystyle{#1}}

\newcommand{\wt}[1]{\widetilde{#1}}

\newcommand{\boldtitle}[1]{\title{\bfseries #1}}

\newenvironment{MSC}{%
\smallbreak
\noindent \textbf{2010\ Mathematics Subject Classification\,:}}

\newenvironment{keywords}{%
\noindent\textbf{Key words and phrases\,:}\itshape}

\newenvironment{Acknowledgement}{%
\noindent\textit{Acknowledgement.}}

\theoremstyle{theorem}
\newtheorem*{multitheorem}{\variable@name}

\theoremstyle{definition}
\newcommand{\variable@name}{Theorem}
\newtheorem*{multiproclaim}{\variable@name}

\theoremstyle{plain}
\newtheorem{thm}{Theorem}[section]
\newtheorem{prop}[thm]{Proposition}
\newtheorem{lem}[thm]{Lemma}
\newtheorem{cor}[thm]{Corollary}

\theoremstyle{definition}

\newtheorem{example}[thm]{Example}
\newtheorem{remark}[thm]{Remark}

\boldtitle{Milnor-Selberg zeta functions and zeta regularizations}
\author{
 Nobushige KUROKAWA
 Masato WAKAYAMA\thanks{Partially supported by Grant-in-Aid for Scientific Research (B) No. 21340011.} 
 \ and
 Yoshinori YAMASAKI\thanks{Partially supported by Grant-in-Aid for Young Scientists (B) No. 21740019.}
} 
\date{\today}

\begin{document}

\setlength{\baselineskip}{15pt}
\maketitle

\begin{abstract}
 By a similar idea for the construction of Milnor's gamma functions,  
 we introduce ``higher depth determinants'' of the Laplacian
 on a compact Riemann surface of genus greater than one.
 We prove that,
 as a generalization of the determinant expression of the Selberg zeta function,
 this higher depth determinant can be expressed as a product of multiple gamma functions and
 what we call a Milnor-Selberg zeta function.
 It is shown that the Milnor-Selberg zeta function 
 admits an analytic continuation, a functional equation and, 
 remarkably, has an Euler product.
\begin{MSC}
 {\it Primary} 11M36,
 {\it Secondary} 11F72.
\end{MSC} 
\end{abstract}

\tableofcontents

\section{Introduction}
\label{sec:introduction}

 In 1983, Milnor (\cite{M}) introduced a family of functions $\{\g_{r}(z)\}_{r\ge 1}$
 what he thought as a kind of ``higher depth gamma functions''.
 These are defined as partial derivatives of the Hurwitz zeta function
 $\z(w,z):=\sum_{n=0}^\infty (n+z)^{-w}$ at non-positive integer points
 with respect to the variable $w$.
 We call $\g_{r}(z)$ a Milnor gamma function of depth $r$
 and will denote it by $\mG_r(z)$
 (see \cite{KOW2} for several analytic properties of $\mG_r(z)$).
 The purpose of Milnor's study about such functions is to construct functions 
 satisfying a modified version of the Kubert identity
 $f(x)=m^{s-1}\sum^{m-1}_{k=0}f(\frac{x+k}{m})$ (\cite{Kubert}),
 which plays an important role in the study of Iwasawa theory.

 The aim of the present paper is, 
 as an analogue of the study of the Milnor gamma functions, 
 to study ``higher depth determinants''
 of the Laplacians on compact Riemann surfaces
 of genus $g\ge 2$ with negative constant curvature.
 Let us give the definition of the higher depth determinant in more general situations.
 Let $A$ be a linear operator on some space.
 We assume that $A$ has only discrete spectrum 
 with eigenvalues $0=\l_0<\l_1\le\l_2\le\cdots\to \infty$.
 Define a spectral zeta function $\z_{A}(w,z)$ of Hurwitz's type by
\[
 \z_{A}(w,z):=\sum^{\infty}_{j=0}(\l_j+z)^{-w}.
\]
 We further assume that the series converges absolutely and uniformly for $z$ on any compact set
 in some right half $w$-plane, 
 and can be continued meromorphically to a region containing $w=1-r$, for $r\in\bN$.
 Moreover, we assume that it is holomorphic at $w=1-r$.
 In such a situation, we define a {\it higher depth determinant} of $A$ of depth $r$ by
\[
  \Det_{r}(A+z):=
 \exp\Bigl(-\frac{\p}{\p w}\z_{A}(w,z)\Bigl|_{w=1-r}\Bigl).
\]
 When $r=1$, this gives the usual (zeta-regularized) determinant of $A$.
 Notice that if $A$ is a finite-rank operator and
 $\l_1,\ldots,\l_{N}$ are their eigenvalues,
 then we have $\Det_r(A)=\prod^{N}_{j=1}\l_{j}^{\l_{j}^{r-1}}$,
 whence $\Det_{r}(A\oplus B)=(\Det_{r}A)\cdot(\Det_{r}B)$
 if both $A$ and $B$ are finite-rank.

 To state our main results, let us recall 
 the case of the determinant of the Laplacian, that is, the case $r=1$.
 Let $\bH$ be the complex upper half plane with the Poincar\'e metric
 and $\G$ a discrete, co-compact torsion-free subgroup of $SL_2(\bR)$.
 Then, $\G\backslash\bH$ becomes a compact Riemann surface of genus $g\ge 2$.
 Let $\Delta_{\G}=-y^2(\frac{\p^2}{\p x^2}+\frac{\p^2}{\p y^2})$
 be the Laplacian on $\G\backslash\bH$, 
 $\l_j$ the $j$th eigenvalue of $\Delta_{\G}$ and put 
 $\Spec{(\Delta_{\G})}:=\{\l_j\,|\,j\in\bN_{0}\}$. 
 We write $\l_j=r_j^2+\frac{1}{4}$ where $r_j\in i\bR_{>0}$
 if $0\le \l_j< \frac{1}{4}$ and $r_j\ge 0$ otherwise.
 Moreover, let $\a^{\pm}_{j}:=\frac{1}{2}\pm ir_j$.
 Notice that 
 $\a^{\pm}_j\in[0,1]$ with $\a_j^{+}<\a_j^{-}$ if $0\le \l_j<\frac{1}{4}$
 and $\Re(\a^{\pm}_j)=\frac{1}{2}$ otherwise.
 It is shown that the series $\z_{\Delta_{\G}}(w,z)$ 
 converges absolutely for $\Re(w)>1$, 
 admits a meromorphic continuation to the whole plane $\bC$ and
 is in particular holomorphic at $w=0$ (see, e.g., \cite{DHP,Sa,V}). 
 Moreover, the determinant $\det(\Delta_{\G}-s(1-s)):=\Det_{1}(\Delta_{\G}-s(1-s))$
 can be calculated as 
\begin{align}
\label{for:determinant}
 \det\bigl(\Delta_{\G}-s(1-s)\bigl)=G_{\G}(s)Z_{\G}(s)=\phi(s)^{g-1}Z_{\G}(s).
\end{align}
 Here, $\phi(s)$ is a meromorphic function defined by 
\begin{align*}
 \phi(s):
&=e^{-2(s-\frac{1}{2})^2-4(s-\frac{1}{2})\z'(0)+4\z'(-1)}\G(s)^{-2}G(s)^{-4}\\
&=e^{-2(s-\frac{1}{2})^2}\G_1(s)^{-2}\G_2(s)^4
\end{align*}
 with $\z(s)$, $\G(s)$, $G(s)$ and $\G_n(s)$
 being the Riemann zeta function, the classical gamma function,
 the Barnes $G$-function ($=$ a double gamma function (\cite{Barnes1899})) 
 and the Barnes multiple gamma function (\cite{Barnes}), respectively. 
 The other factor $Z_{\G}(s)$ in \eqref{for:determinant} is the Selberg zeta function
 defined by the Euler product
\[
 Z_{\G}(s)
:=\prod_{P\in\Prim(\G)}\prod^{\infty}_{n=0}\bigl(1-N(P)^{-s-n}\bigr)
 \qquad (\Re(s)>1).
\]
 Here, $\Prim(\G)$ is the set of all primitive hyperbolic conjugacy classes of $\G$ and
 $N(P)$ is the square of the larger eigenvalue of $P\in\Prim(\G)$.
 It is known that $Z_{\G}(s)$ can be continued analytically to the whole plane $\bC$ and 
 has the functional equation
\begin{equation}
\label{for:fe1}
 Z_{\G}(1-s)=\Bigl(S_1(s)^{2}S_2(s)^{-4}\Bigr)^{g-1}Z_{\G}(s),
\end{equation} 
 where $S_n(s)$ is the normalized multiple sine function
 defined by \eqref{def:multsine} (\cite{KK1}). 
 Notice that if we define the complete Selberg zeta function $\Xi_{\G}(s)$ by
\[
 \Xi_{\G}(s):=\bigl(\G_2(s)^2\G_2(s+1)^2\bigr)^{g-1}Z_{\G}(s),
\] 
 then, the functional equation \eqref{for:fe1} is equivalent to 
\begin{equation}
\label{for:fe2}
 \Xi_{\G}(1-s)=\Xi_{\G}(s).
\end{equation} 
 Moreover, it is known that $Z_{\G}(s)$ has zeros at $s=1$, $0$, $-k$ for $k\in\bN$ 
 with multiplicity $1$, $2g-1$, $2(g-1)(2k+1)$, respectively 
 (the latest are called the trivial zeros
 because they also come from the gamma factor as the Riemann zeta function)
 and at $s=\a^{\pm}_j$ for $j\in\bN$
 (these are called the non-trivial zeros).
 In particular, 
 $Z_{\G}(s)$ satisfies an analogue of the Riemann hypothesis,
 i.e., all imaginary zeros of $Z_{\G}(s)$ are located on the line $\Re(s)=\frac{1}{2}$.
 See \cite{Deninger1992} for arithmetic trial of
 a determinant expression for the Riemann zeta function.
 
 In this paper, as generalizations of the case of $r=1$,
 we study the function 
\begin{align*}
 \mathrm{D}_{\G,r}(s)
:=\Det_r\bigl(\Delta_{\G}-s(1-s)\bigl)
:=\exp\Bigl(-\frac{\p}{\p w}\z_{\Delta_{\G}}\bigl(w,-s(1-s)\bigr)\Bigl|_{w=1-r}\Bigl)
\end{align*}
 for an arbitrary $r\in\bN$.
 Here, for $j$ corresponding to the zero and the exceptional eigenvalues $\l_j$
 (that is, $0\le \l_j<\frac{1}{4}$),
 the summand $(\l_j-s(1-s))^{-w}$
 of the spectral zeta function $\z_{\Delta_{\G}}\bigl(w,-s(1-s)\bigr)$
 is defined by $(\l_j-s(1-s))^{-w}:=\exp(-w\log^{(j)}(\l_j-s(1-s)))$ 
 where $\log^{(j)}$ is one of the branch of the logarithm defined by
 \eqref{def:branch} in Section~\ref{sec:HD}.
 For the other $j$, as usual, we employ the principal branch of the logarithm $\log$
 which takes values in $\bR+i(-\pi,\pi]$.
 We will see in Section~\ref{sec:HD} that 
 the defining domain of $\mathrm{D}_{\G,r}(s)$ is a certain region
 with two connected components
 which are convex and symmetric with respect to the line $\Re(s)=\frac{1}{2}$.
 
Define the region $\Omega_{\G}$ by
\begin{align}
\label{def:omega}
 \Omega_{\G}
:=\bC\setminus\Biggl(\biggl(&\bigcup_{0\le \l_j<\frac{1}{4}}
[\a^{+}_j,\a^{+}_j-i\infty)\cup [\a^{-}_j,\a^{-}_j+i\infty)\biggr)\\
&\ \ \ \cup\biggl(\bigcup_{\l_j\ge \frac{1}{4}}
\bigl[\a^{+}_{j},\frac{1}{2}+i\infty\bigr)\cup \bigl[\a^{-}_{j},\frac{1}{2}-i\infty\bigr)\biggr) \Biggr).
\nonumber
\end{align}
 Here, for $\a,\b\in\bC$,
 the semi-closed line from $\a$ to $\b$ are denoted by $[\a,\b)$ and so on.
 Notice that $\Omega_{\G}$ is connected if and only if
 $\frac{1}{4}\notin\Spec{(\Delta_{\G})}$ (see Figure~$1$).
 The following theorem describes main results of this paper.

\begin{figure}
 \psfrag{Re}{\footnotesize $\Re$}
 \psfrag{Im}{\footnotesize $\Im$}
 \psfrag{A}{\footnotesize $0$}
 \psfrag{B}{\footnotesize $1$}
 \psfrag{C}{\footnotesize $\frac{1}{2}$}
 \psfrag{S}{\footnotesize $\a_j^{+}$}
 \psfrag{T}{\footnotesize $\a_j^{-}$}
\begin{center}
  \includegraphics[clip,width=70mm]{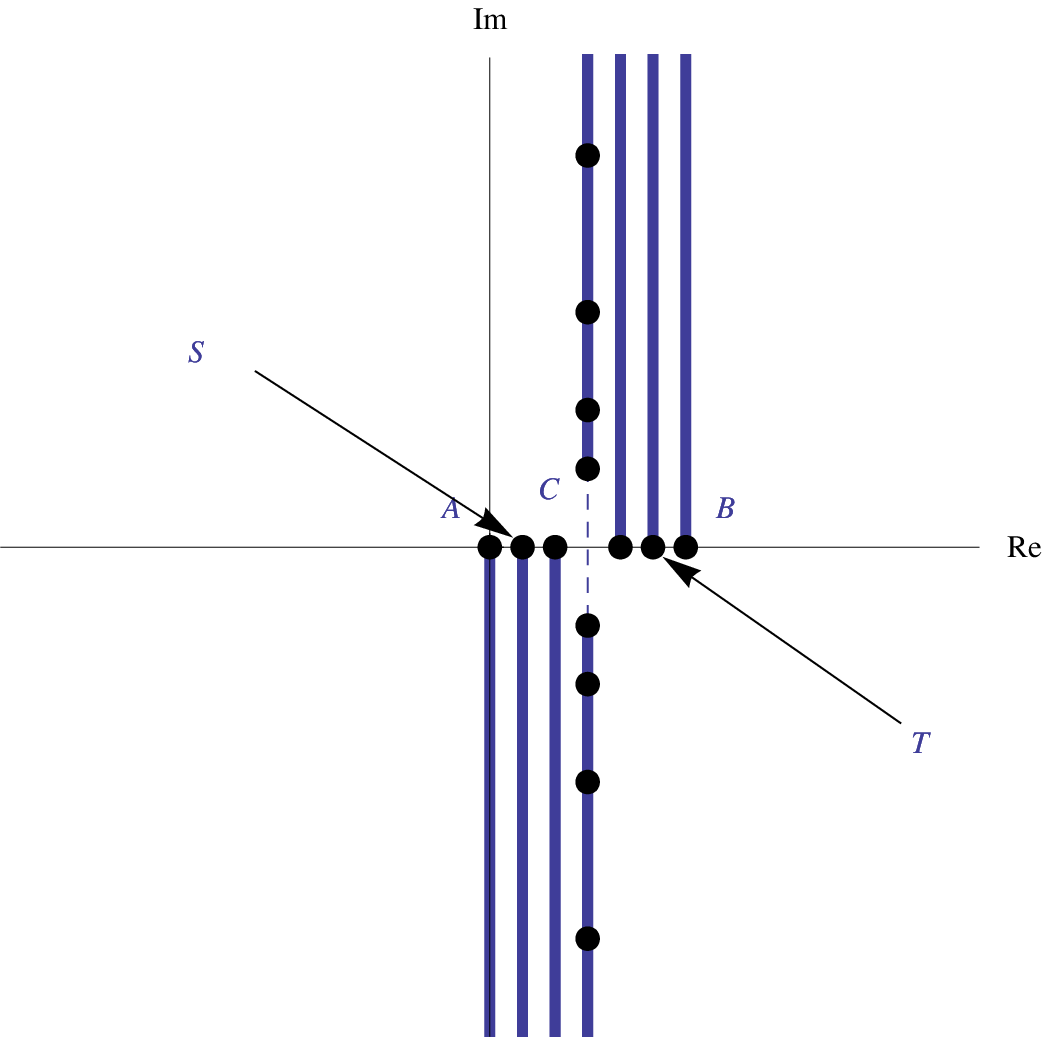}
     \caption{$\mathrm{D}_{\G,r}(s)$ is holomorphic in $\Omega_{\G}$
 ($\frac{1}{4}\notin\Spec{(\Delta_{\G})}$).}
\end{center}

\end{figure}
\begin{figure}[htbp]
\begin{center}
  \begin{tabular}{cc}
   \begin{minipage}{0.5\textwidth}
 \psfrag{Re}{\footnotesize $\Re$}
 \psfrag{Im}{\footnotesize $\Im$}
 \psfrag{A}{\footnotesize $0$}
 \psfrag{D}{\footnotesize $-1$}
 \psfrag{E}{\footnotesize $-2$}
    \begin{center}
     \includegraphics[clip,width=70mm]{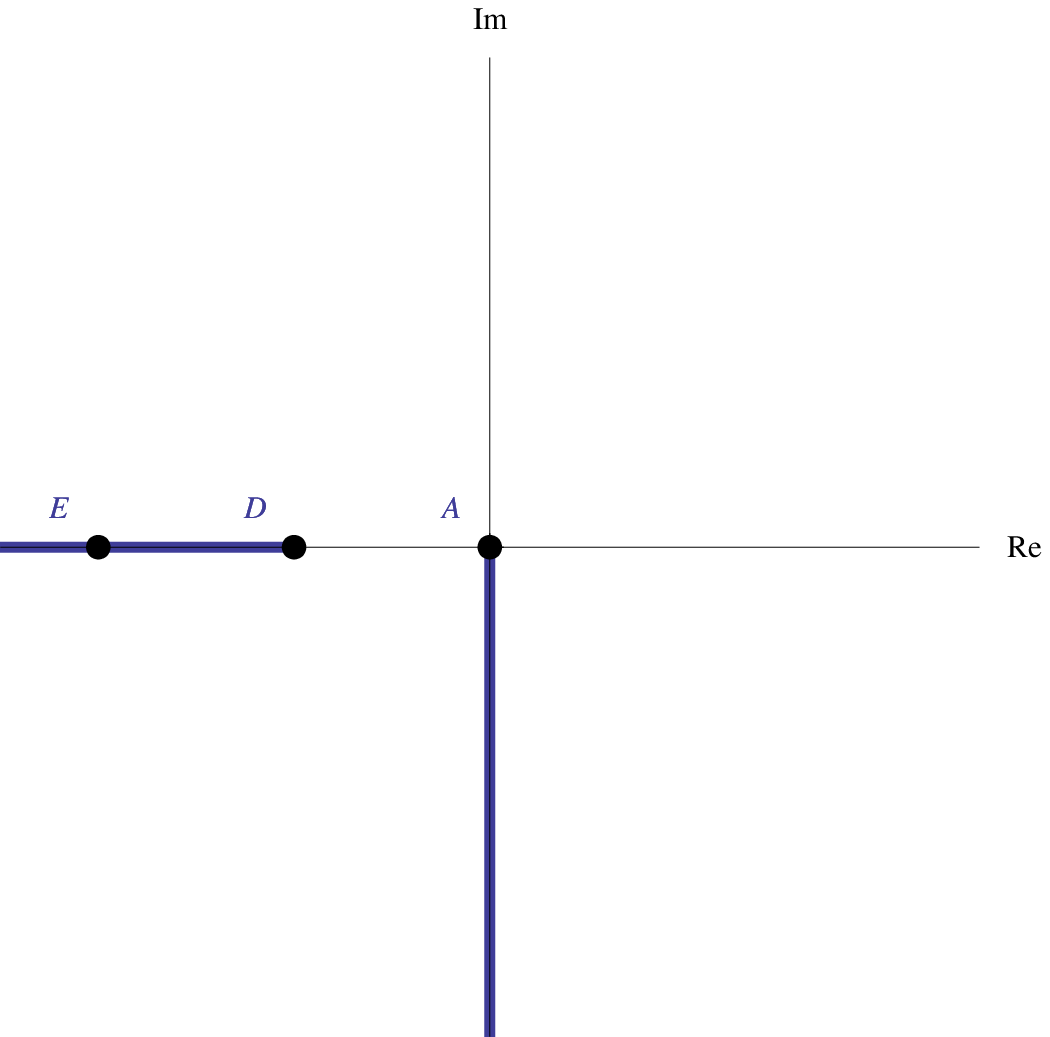}
     \caption{$\phi_r(s)$ is holomorphic in
 $\bC\setminus \bigl((-\infty,-1]\cup [0,-i\infty)\bigr)$.}
    \end{center}
   \end{minipage}
&
   \begin{minipage}{0.5\textwidth}
 \psfrag{Re}{\footnotesize $\Re$}
 \psfrag{Im}{\footnotesize $\Im$}
 \psfrag{A}{\footnotesize $0$}
 \psfrag{C}{\footnotesize $\frac{1}{2}$}
 \psfrag{B}{\footnotesize $1$}
 \psfrag{D}{\footnotesize $-1$}
 \psfrag{E}{\footnotesize $-2$} 
 \psfrag{S}{\footnotesize $\a_j^{+}$}
 \psfrag{T}{\footnotesize $\a_j^{-}$}
    \begin{center}
     \includegraphics[clip,width=70mm]{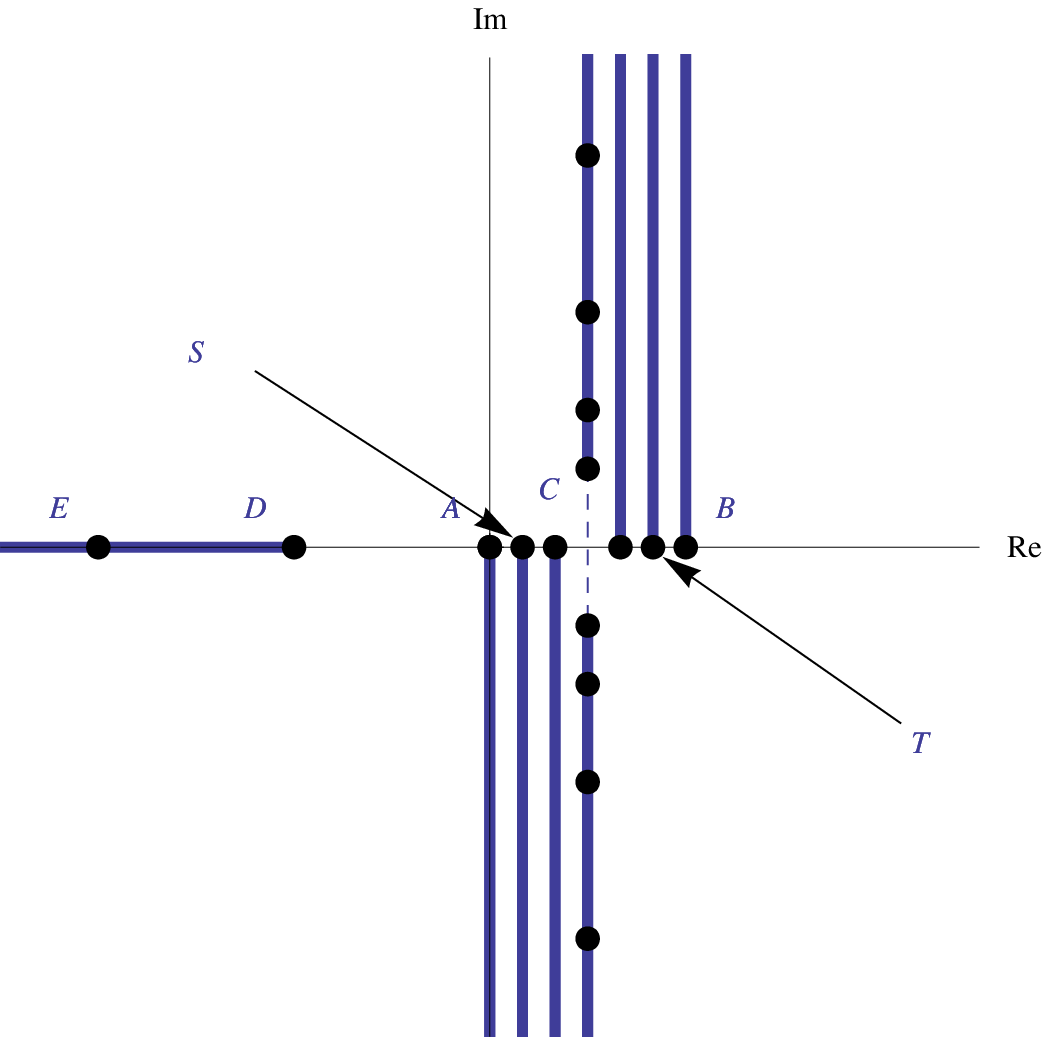}
     \caption{$Z_{\G,r}(s)$ is holomorphic in $\Omega_{\G}\setminus (-\infty,-1]$
 ($\frac{1}{4}\notin\Spec{(\Delta_{\G})}$).}
    \end{center}
\ \\[-35pt]
   \end{minipage}
  \end{tabular}
\end{center}
\end{figure}

\begin{thm}
\label{thm:main}
 Let $r\ge 2$. 

 $(\mathrm{i})$ The function $\mathrm{D}_{\G,r}(s)$
 can be extended holomorphically to the region $\Omega_{\G}$ 
 and satisfies $\mathrm{D}_{\G,r}(1-s)=\mathrm{D}_{\G,r}(s)$.

 $(\mathrm{ii})$ It can be written as 
\begin{align}
\label{for:main+}
 \mathrm{D}_{\G,r}(s)
=G_{\G,r}(s)Z_{\G,r}(s)=\phi_r(s)^{g-1}Z_{\G,r}(s).
\end{align}
 Here,
 $\phi_r(s)$ is a holomorphic function in $\bC\setminus \bigl((-\infty,-1]\cup [0,-i\infty)\bigr)$
 expressed as \eqref{for:gammafactor} by the Milnor gamma functions and, further,
 can be written as a product of multiple gamma functions.
 The other factor $Z_{\G,r}(s)$, which we call a Milnor-Selberg zeta function of depth $r$,
 is a holomorphic function in $\Omega_{\G}\setminus (-\infty,-1]$ 
 having a functional equation 
\begin{equation}
\label{for:fer}
 Z_{\G,r}(1-s)
=\Bigl(\prod^{2r}_{j=1}S_j(s)^{-\a_{r,j}(s-\frac{1}{2})}\Bigr)^{g-1}Z_{\G,r}(s)
\end{equation} 
 for $s\in\Omega_{\G}\setminus \bigl((-\infty,-1]\cup [2,+\infty)\bigr)$
 where $\a_{r,j}(t)$ is the polynomial defined in \eqref{def:alpha}.
 Moreover, $Z_{\G,r}(s)$ has an Euler product expression in a sense that 
 it can be written as a product and quotient of ``poly-Selberg zeta functions'' $Z^{(m)}_{\G}(s)$
 defined by an Euler product for $\Re(s)>1$.
\end{thm}

 Similarly to the result of $r=1$, 
 we may call $G_{\G,r}(s)$ (or $\phi_r(s)$) a gamma factor.
 Note that $\phi_r(s)$ and $Z_{\G,r}(s)$ are
 generalizations of $\phi(s)$ and $Z_{\G}(s)$, respectively.
 Actually, one can see that $\phi_1(s)=\phi(s)$, $Z_{\G,1}(s)=Z_{\G}(s)$,
 $\a_{1,1}(t)=-2$ and $\a_{1,2}(t)=4$ (see Remark~\ref{rem:constant}) and hence
 the equation \eqref{for:main+} (resp. \eqref{for:fer}) coincides with
 \eqref{for:determinant} (resp. \eqref{for:fe1}) when $r=1$.

 We remark that, from the equation \eqref{for:main+},   
 remarkable cancellation of the singularities on $(-\infty,-1]$ occurs 
 (though both functions $G_{\G,r}(s)$ and $Z_{\G,r}(s)$ have singularities on the above interval
 (see Figure~$2$ and $3$),
 their product $\mathrm{D}_{\G,r}(s)$ does not). 
 Visit also Remark~\ref{rem:multpoly1} for some observations on these singularities. 

 The organization of the paper is as follows.
 In Section~\ref{sec:HD},
 from the integral representation of the spectral zeta function,
 we first show the existence of $\mathrm{D}_{\G,r}(s)$ and, then, study its basic analytic properties. 
 In particular,
 we give a proof of the claim $(\mathrm{i})$ in Theorem~\ref{thm:main} (Theorem~\ref{thm:HDDs}).
 Here, we make a special choice \eqref{def:branch} of the log-branch
 which allows us reasonably to reach the functional equation of $\mathrm{D}_{\G,r}(s)$
 (see Remark~\ref{rem:strict}).
 Next, using the Selberg trace formula, we derive the factorization \eqref{for:main+}.
 In Section~\ref{sec:gamma factor},
 we determine the explicit expression of the gamma factor $\phi_r(s)$   
 in terms of the Milnor gamma function $\mG_r(z)$ (Proposition~\ref{prop:gammafactor}). 
 Furthermore, we give two other expressions of $\phi_r(s)$; one by 
 the Barnes multiple gamma functions $\G_n(z)$
 (Theorem~\ref{thm:gamma factor1}), and another by 
 the Vign\'eras multiple gamma functions $G_n(z)$
 (Corollary~\ref{cor:gamma factor2}).  
 We notice that the latter two take advantage of observing analytic properties of $\phi_r(s)$.
 Using ladder relations of the multiple gamma and sine functions, 
 we also have a functional equation of $\phi_r(s)$ (Theorem~\ref{thm:fegamma}).
 In Section~\ref{sec:zeta factor}, via the expression \eqref{for:main+}, 
 we study analytic properties of the Milnor-Selberg zeta
 function $Z_{\G,r}(s)$ such as
 an analytic continuation and a functional equation (Theorem~\ref{thm:analyticpropZ}).
 Moreover, introducing a ``poly-Selberg zeta function'' $Z^{(m)}_{\G}(s)$
 by a certain Euler product (which is regarded as another generalization of the Selberg zeta function), 
 we show that  $Z_{\G,r}(s)$ can be expressed as a product and quotient of
 $Z_{\G}^{(m)}(s)$ (Theorem \ref{thm:zeta factor}).
 This is nothing but the Euler product expression of $Z_{\G,r}(s)$.

 In the course of the explicit determination of the gamma factor $\phi_r(s)$,
 we will encounter the following series involving the Hurwitz zeta function;
\[
  R_m(t,z):=\sum^{\infty}_{j=1}\frac{\z(2j+1,z)}{2j+m+1}t^{2j+m+1}.
\]
 This type of series has been studied in several places
 (see, e.g., \cite{A,KS1}). 
 It will be shown that the exponential of
 such a series (with $z=\frac{1}{2}$) is expressed as a product
 of the Milnor gamma function (Proposition~\ref{prop:Pa}).
 We also note that a similar discussion developed in this paper yields 
 explicit expressions of the higher depth determinants of the Laplacian 
 on the higher dimensional spheres (\cite{Yamasaki}).
 See also \cite{WakayamaYamasaki} for number theoretic analogues of the present study. 

 In this paper, as usual, $\bC$, $\bR$, $\bQ$ are respectively denoted by the field of
 all complex, real and rational numbers.
 We also use the notations $\bZ$, $\bN$ and $\bN_0$ to denote the set of
 all rational, positive and non-negative integers, respectively.  

\section{Higher depth determinants}
\label{sec:HD}

\subsection{Existence and basic properties}

 Let $\Delta_{\G}$ be the Laplacian
 on the compact Riemann surface $\G\backslash\bH$ of genus $g\ge 2$,
 $0=\l_0<\l_1\le \l_2\le\cdots\to\infty$ the eigenvalues of $\Delta_{\G}$ 
 and $\Spec(\Delta_{\G}):=\{\l_j\,|\,j\in\bN_{0}\}$. 
 For $T\ge 0$, define $\theta^{(T)}_{\Delta_{\G}}(t):=\sum_{\l_j\ge T}e^{-\l_jt}$.
 It is known that $\theta^{(T)}_{\Delta_{\G}}(t)$ converges absolutely for $\Re(t)>0$ and 
 has the asymptotic formula
 $\theta^{(T)}_{\Delta_{\G}}(t)\sim\sum^{\infty}_{n=-1}c^{(T)}_{n}t^{n}$
 as $t\downarrow 0$ with $c^{(T)}_{-1}\ne 0$
 (see, e.g., \cite{Rosenberg1997}).
 Notice that $c^{(T)}_{-1}$ is not depend on $T$ and hence we denote it by $c_{-1}$.
 Define $\z^{(T)}_{\Delta_{\G}}(w,z):=\sum_{\l_j\ge T}(\l_j+z)^{-w}$ 
 where we let $(\l_j+z)^{-w}:=\exp(-w\log{(\l_j+z)})$ for all $j\in\bN_0$.
 In particular, we put $\z_{\Delta_{\G}}(w,z):=\z^{(0)}_{\Delta_{\G}}(w,z)$.
 The series $\z^{(T)}_{\Delta_{\G}}(w,z)$ converges absolutely and uniformly
 in any compact set in 
 $\{w\in\bC\,|\,\Re(w)>1\}\times\{z\in\bC\,|\,\Re(z)>-T\}$ and hence
 defines a holomorphic function in the region.
 At first, one easily sees the following 

\begin{lem}
\label{lem:T}
 The function $\z^{(T)}_{\Delta_{\G}}(w,z)$ admits a meromorphic continuation to the region 
 $\bC\times\{z\in\bC\,|\,\Re(z)>-T\}$
 with a simple pole at $w=1$ and being regular otherwise.
 In other words, it can be written as 
\[
 \z^{(T)}_{\Delta_{\G}}(w,z)=\frac{c_{-1}}{w-1}+\psi^{(T)}_{\Delta_{\G}}(w,z),
\]
 where $\psi^{(T)}_{\Delta_{\G}}(w,z)$ is a holomorphic function in $\bC\times\{z\in\bC\,|\,\Re(z)>-T\}$.
\end{lem}
\begin{proof}
 Let $\Re(w)>1$ and $\Re(z)>-T$.
 Then, from the integral expression 
\begin{equation}
\label{for:intspectral}
 \z^{(T)}_{\Delta_{\G}}(w,z)
=\frac{1}{\G(w)}\int^{\infty}_{0}t^we^{-zt}\theta^{(T)}_{\Delta_{\G}}(t)\frac{dt}{t},
\end{equation}
 for any $N\in\bN_0$, we have 
\begin{align*}
 \z^{(T)}_{\Delta_{\G}}(w,z)
&=\frac{1}{\G(w)}\int^{1}_{0}t^we^{-zt}
\Bigl(\theta^{(T)}_{\Delta_{\G}}(t)-\sum^{N}_{n=-1}c^{(T)}_{n}t^{n}\Bigr)\frac{dt}{t}\nonumber\\
&\ \ \
 +\frac{1}{\G(w)}\int^{1}_{0}t^we^{-zt}\Bigl(\sum^{N}_{n=-1}c^{(T)}_{n}t^{n}\Bigr)\frac{dt}{t}
+\frac{1}{\G(w)}\int^{\infty}_{1}t^we^{-zt}\theta^{(T)}_{\Delta_{\G}}(t)\frac{dt}{t}\nonumber\\
&=\frac{1}{\G(w)}\int^{1}_{0}t^we^{-zt}
\Bigl(\theta^{(T)}_{\Delta_{\G}}(t)-\sum^{N}_{n=-1}c^{(T)}_{n}t^{n}\Bigr)\frac{dt}{t}\nonumber\\
&\ \ \
 +\frac{1}{\G(w)}\sum^{N}_{n=-1}c^{(T)}_{n}\int^{1}_{0}t^{w+n-1}
\Bigl(\sum^{\infty}_{m=0}\frac{(-zt)^m}{m!}\Bigr)dt
+\frac{1}{\G(w)}\int^{\infty}_{1}t^we^{-zt}\theta^{(T)}_{\Delta_{\G}}(t)\frac{dt}{t},\nonumber
\end{align*}
 whence 
\begin{align}
\label{for:ac_speczeta}
 \z^{(T)}_{\Delta_{\G}}(w,z)
&=\frac{1}{\G(w)}\int^{1}_{0}t^we^{-zt}\Bigl(\theta^{(T)}_{\Delta_{\G}}(t)-\sum^{N}_{n=-1}c^{(T)}_{n}t^{n}\Bigr)\frac{dt}{t}\\
&\ \ \ +\frac{1}{\G(w)}\sum^{N}_{n=-1}\sum^{\infty}_{m=0}\frac{c^{(T)}_{n}(-z)^m}{m!(w+n+m)}+\frac{1}{\G(w)}\int^{\infty}_{1}t^we^{-zt}\theta^{(T)}_{\Delta_{\G}}(t)\frac{dt}{t}.\nonumber
\end{align}
 Since $\theta^{(T)}_{\Delta_{\G}}(t)-\sum^{N}_{n=-1}c^{(T)}_{n}t^{n}=O(t^{N+1})$
 as $t\downarrow 0$,
 the first integral in the righthand-side of \eqref{for:ac_speczeta}
 converges absolutely for $\Re(w)>-(N+1)$ and hence, as a function of $w$,
 defines a holomorphic function in the region.
 The second term defines a meromorphic function on $\bC$ having a simple pole at $w=1$
 (notice that the points $w=0,-1,-2,\ldots$ are not poles of $\z^{(T)}_{\Delta_{\G}}(w,z)$
 because of the gamma factor).
 Moreover, since $\Re(z)>-T$,
 the last integral converges absolutely for all $w\in\bC$, 
 whence it defines an entire function.
 Therefore, letting $N\to\infty$, 
 we obtain a meromorphic continuation of $\z^{(T)}_{\Delta_{\G}}(w,z)$
 to $\bC\times\{z\in\bC\,|\,\Re(z)>-T\}$.
\end{proof}

 We now study the higher depth determinants.
 From Lemma~\ref{lem:T}, the function
\[
 \Det^{(T)}_{r}\bigl(\Delta_{\G}+z\bigr)
:=\exp\Bigl(-\frac{\p}{\p w}\z^{(T)}_{\Delta_{\G}}(w,z)\Bigl|_{w=1-r}\Bigr)
\]
 is well-defined and is holomorphic in $\{z\in\bC\,|\,\Re(z)>-T\}$.
 Notice that, since $\frac{\p}{\p w}\z^{(T)}_{\Delta_{\G}}(w,z)\bigl|_{w=1-r}$ is holomorphic, 
 by the definition, $\Det^{(T)}_{r}(\Delta_{\G}+z)$ has no zeros in the region.
 In particular, we put $\Det_r(\Delta_{\G}+z):=\Det^{(0)}_r(\Delta_{\G}+z)$.

\begin{prop}
\label{prop:HDDz}
 The function $\Det_r(\Delta_{\G}+z)$ can be continued analytically to 
\ \\[-20pt]
\begin{itemize}
 \item[$\mathrm{(i)}$] an entire function with zeros at $z=-\l_j$ when $r=1$.\ \\[-20pt]
 \item[$\mathrm{(ii)}$] a holomorphic function in $\bC\setminus (-\infty,0]$ 
 with no zeros when $r\ge 2$. \ \\[-20pt]
\end{itemize}
\end{prop}
\begin{proof}
 Let $\Re(z)>0$.
 Then, for any $T>0$, we have  
\begin{equation*}
 \z_{\Delta_{\G}}(w,z)=\sum_{0\le \l_j<T}(\l_j+z)^{-w}+\z^{(T)}_{\Delta_{\G}}(w,z).
\end{equation*}
 From Lemma~\ref{lem:T}, the righthand-side is holomorphic at $w=1-r$.
 Hence, differentiating both sides at $w=1-r$, we have
\begin{align}
\label{for:methodAC}
 \Det_r\bigl(\Delta_{\G}+z\bigr)
&=\prod_{0\le\l_j<T}\exp\Bigl((\l_j+z)^{r-1}\log(\l_j+z)\Bigr)\cdot
 \Det^{(T)}_r\bigr(\Delta_{\G}+z\bigr).
 \end{align}
 When $r=1$, since $\exp(\log{(\l_j+z)})=\l_j+z$,
 the first factor in the righthand-side of \eqref{for:methodAC} is a polynomial
 and hence is entire.
 This shows that $\Det_1(\Delta_{\G}+z)$ can be extended analytically
 to $\{z\in\bC\,|\,\Re(z)>-T\}$.
 When $r\ge 2$, the first factor in this case is holomorphic in
 $\bigcap_{0\le \l_j<T}\bC\setminus (-\infty,-\l_j]=\bC\setminus (-\infty,0]$,
 whence $\Det_r(\Delta_{\G}+z)$ can be extended to $\{z\in\bC\,|\,\Re(z)>-T\}\setminus(-T,0]$.
 Therefore, letting $T\to\infty$, one obtains the desired claim. 
\end{proof}

 Based on the above discussions, we next study the case $z=-s(1-s)$.
 Write $\l_j=r_j^2+\frac{1}{4}$ where $r_j\in i\bR_{>0}$
 if $0\le \l_j<\frac{1}{4}$ and $r_j\ge 0$ otherwise.
 Moreover, let $\a^{\pm}_{j}:=\frac{1}{2}\pm ir_j$.
 We notice that $\l_j-s(1-s)=(s-\a^{+}_j)(s-\a^{-}_j)$
 and $\a^{\pm}_j\in[0,1]$ with $\a_j^{+}<\a_j^{-}$ if $0\le \l_j<\frac{1}{4}$
 and $\Re(\a^{\pm}_j)=\frac{1}{2}$ otherwise.
 As is the case of the previous discussion, 
 we also start from the zeta function
 $\z^{(T)}_{\Delta_{\G}}(w,-s(1-s))=\sum_{\l_j\ge T}(\l_j-s(1-s))^{-w}$,
 however, in this case, for small $j$, we replace the branch of the logarithm.
 Namely, we let $(\l_j-s(1-s))^{-w}:=\exp(-w\log^{(j)}(\l_j-s(1-s)))$ where
\begin{align}
\label{def:branch}
 \log^{(j)}(\l_j-s(1-s))
&:=\log{|\l_j-s(1-s)|}\\
&\ \ \ +i
\begin{cases}
 \bigl(\arg_{+}(s-\a^{+}_j)+\arg_{-}(s-\a^{-}_j)\bigr) & (0\le \l_j<\frac{1}{4}), \\
 \arg\bigl((s-\a^{+}_j)(s-\a^{-}_j)\bigr) & (\l_j\ge \frac{1}{4})
\end{cases}
\nonumber
\end{align}
 with the argument $\arg_{\pm}$ being respectively taken as
 $-\frac{1}{2}\pi\le \arg_{+}{z}<\frac{3}{2}\pi$ and 
 $-\frac{3}{2}\pi\le \arg_{-}{z}<\frac{1}{2}\pi$ for $z\in\bC$.
 We notice that $\log^{(j)}=\log$ if $\l_j\ge \frac{1}{4}$. 
 It is easy to see that 
 $l_j(s):=\log^{(j)}(\l_j-s(1-s))$ is a single-valued holomorphic function in $W_j$
 where 
 \[
 W_j:=
\begin{cases}
 \DS{\bC\setminus
 \Bigl(\bigl[\a^{+}_{j},\a^{+}_{j}-i\infty\bigr)\cup\bigl[\a^{-}_{j},\a^{-}_{j}+i\infty\bigr)\Bigr)}
 & (0\le \l_j<\frac{1}{4}), \\
 \DS{\bC\setminus
 \Bigl(\bigl[\a^{+}_{j},\frac{1}{2}+i\infty\bigr)\cup\bigl[\a^{-}_{j},\frac{1}{2}-i\infty\bigr)\Bigr)}
 & (\l_j\ge \frac{1}{4}),
\end{cases}
\]
 and satisfies the trivial functional equation $l_j(1-s)=l_j(s)$ for $s\in W_j$. 
 Let $U^{(T)}:=\{s\in \bC\,|\,\Re(-s(1-s))>-T\}$
 (see Figures~$4$, $5$ and $6$ for $T=0$, $T=\frac{1}{4}$ and $T>\frac{1}{4}$, respectively).
 Note that $U^{(T)}$ has two connected components if $0\le T\le \frac{1}{4}$
 and is connected otherwise.
 Clearly, $\lim_{T\to\infty}U^{(T)}=\bC$.
 Moreover, let $W^{(T)}:=\bigcap_{\l_j>T}W_j$. 
 Notice that $W^{(T)}\supset U^{(T)}$.
 From Lemma~\ref{lem:T},
 the function $\z^{(T)}_{\Delta_{\G}}(w,-s(1-s))$
 is meromorphic in $\bC\times (W^{(T)}\cap U^{(T)})=\bC\times U^{(T)}$
 and is in particular holomorphic at $w=1-r$ for all $r\in\bN$.
 Let $\mathrm{D}^{(T)}_{\G,r}(s):=\Det^{(T)}_{r}(\Delta_{\G}-s(1-s))$,
 which is holomorphic in $U^{(T)}$ and
 satisfies $\mathrm{D}^{(T)}_{\G,r}(1-s)=\mathrm{D}^{(T)}_{\G,r}(s)$ for $s\in U^{(T)}$.
 Define the region $\Omega_{\G}:=\bigcap^{\infty}_{j=0}W_j$.
 Notice that $\Omega_{\G}$ is equal to the righthand-side of \eqref{def:omega}
 and remark that it is connected if and only if $\frac{1}{4}\notin\Spec(\Delta_{\G})$.
 Similarly to the previous discussion, 
 an analytic continuation of $\mathrm{D}_{\G,r}(s):=\mathrm{D}^{(0)}_{\G,r}(s)$
 beyond the region $U:=U^{(0)}$
 is given as follows.
 Let $j_{\G}:=\min\{j\,|\,\l_j\ge \frac{1}{4}\}$ and
 $T>\l_{j_{\G}}(\ge \frac{1}{4})$. 
 For $s\in U$, we have
\begin{align}
\label{for:methodACsdet}
 \mathrm{D}_{\G,r}(s)
&=e^{(T)}_{\G,r}(s)\cdot \mathrm{D}^{(T)}_{\G,r}(s),
\end{align} 
 where 
\[
 e^{(T)}_{\G,r}(s)
:=\prod_{0\le \l_j<T}\exp\Bigl(\bigl(\l_j-s(1-s)\bigr)^{r-1}\log^{(j)}{\bigl(\l_j-s(1-s)\bigr)}\Bigr).
\]
 Note that, by the definition, $e^{(T)}_{\G,r}(s)$
 also satisfies the trivial functional equation $e^{(T)}_{\G,r}(1-s)=e^{(T)}_{\G,r}(s)$.
 When $r=1$, since $e^{(T)}_{\G,1}(s)$ is again a polynomial,
 it is an entire function.
 On the other hand, when $r\ge 2$,
 $e^{(T)}_{\G,r}(s)$ is holomorphic in $\Omega_{\G}$.
 From the equation \eqref{for:methodACsdet},
 these show that $\mathrm{D}_{\G,r}(s)$ can be extended to $U^{(T)}$
 if $r=1$ and $U^{(T)}\cap \Omega_{\G}$ otherwise.
 Therefore, letting $T\to\infty$ and noting that
 $\lim_{T\to\infty}U^{(T)}\cap \Omega_{\G}=\Omega_{\G}$,
 one obtains the following

\begin{thm}
\label{thm:HDDs}
 The function $\mathrm{D}_{\G,r}(s)$ can be continued analytically to 
 \ \\[-20pt]
\begin{itemize}
 \item[$\mathrm{(i)}$] an entire function with zeros at $s=\a^{\pm}_j$ when $r=1$. \ \\[-20pt]
 \item[$\mathrm{(ii)}$] a holomorphic function in $\Omega_{\G}$ with no zeros when $r\ge 2$. \ \\[-20pt]
\end{itemize}
 Moreover, $\mathrm{D}_{\G,r}(s)$ satisfies the functional equation
 $\mathrm{D}_{\G,r}(1-s)=\mathrm{D}_{\G,r}(s)$.
\qed
\end{thm}

\begin{figure}[htbp]
\begin{center}
  \begin{tabular}{ccc}
   \begin{minipage}{0.33\textwidth}
 \psfrag{Re}{\footnotesize $\Re$}
 \psfrag{Im}{\footnotesize $\Im$}
 \psfrag{A}{\footnotesize $0$}
 \psfrag{B}{\footnotesize $\frac{1}{2}$}
 \psfrag{C}{\footnotesize $1$}
    \begin{center}
     \includegraphics[clip,width=40mm]{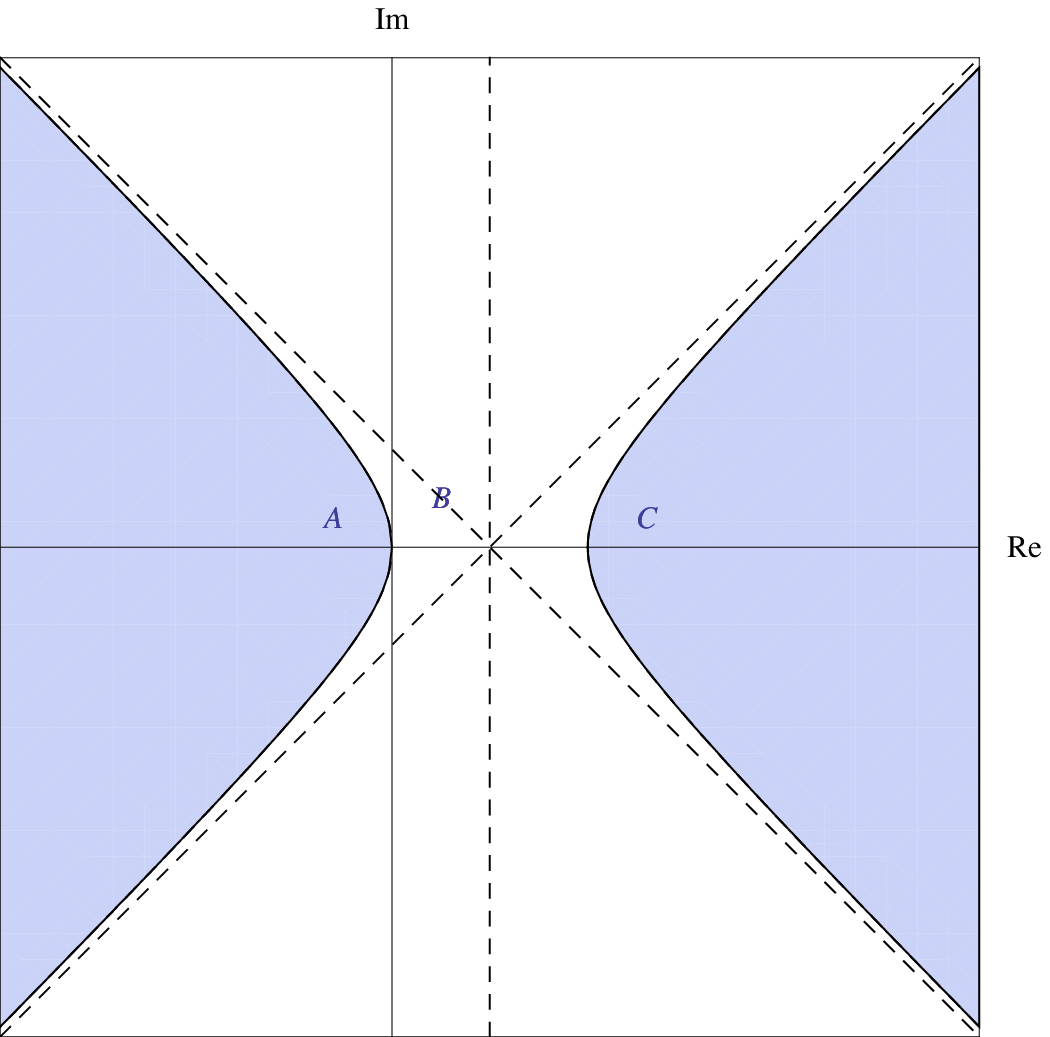}
     \caption{$U=U^{(0)}$.}
    \end{center}
   \end{minipage}
   \begin{minipage}{0.33\textwidth}
 \psfrag{Re}{\footnotesize $\Re$}
 \psfrag{Im}{\footnotesize $\Im$}
 \psfrag{A}{\footnotesize $0$}
 \psfrag{B}{\footnotesize $\frac{1}{2}$}
 \psfrag{C}{\footnotesize $1$}
    \begin{center}
     \includegraphics[clip,width=40mm]{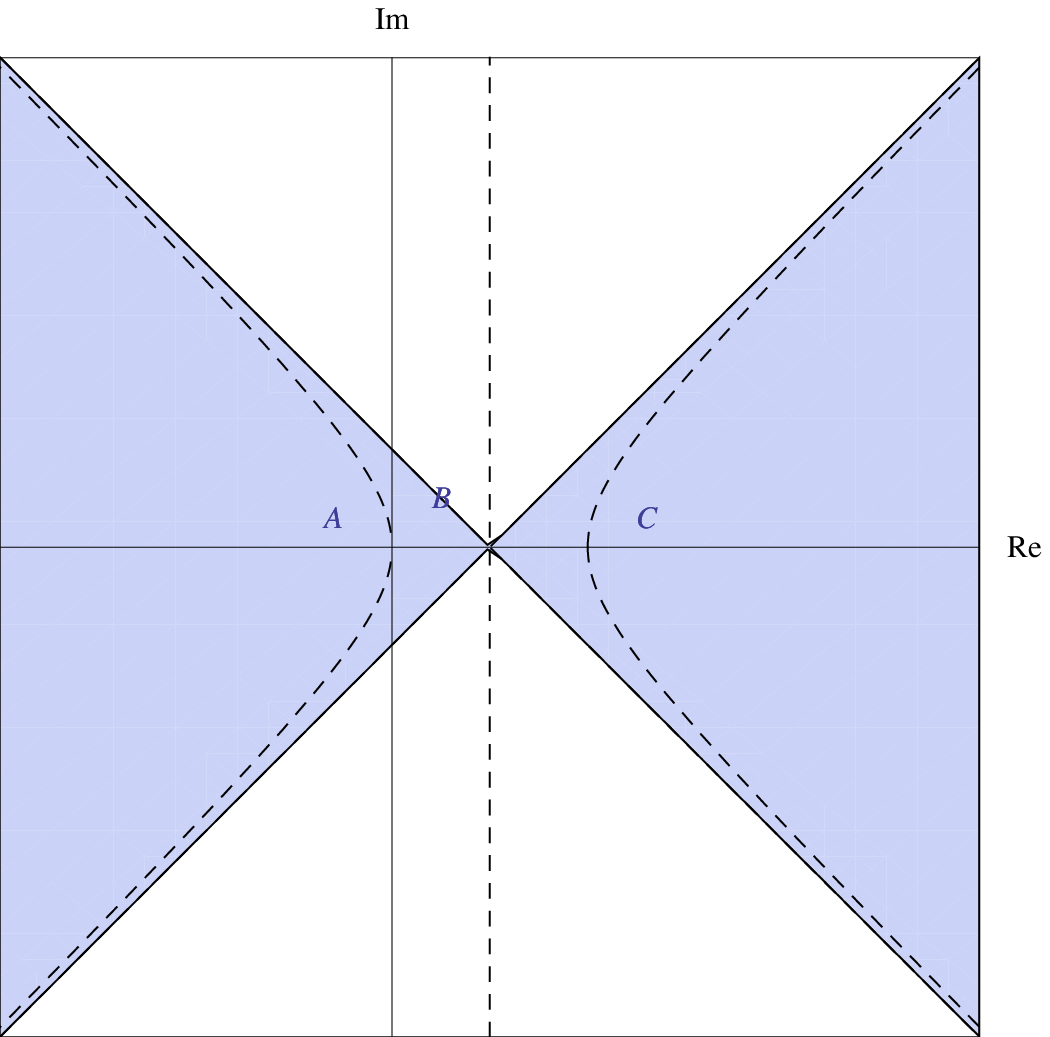}
     \caption{$U^{(\frac{1}{4})}$.}
    \end{center}
   \end{minipage}
  \begin{minipage}{0.33\textwidth}
 \psfrag{Re}{\footnotesize $\Re$}
 \psfrag{Im}{\footnotesize $\Im$}
 \psfrag{A}{\footnotesize $0$}
 \psfrag{B}{\footnotesize $\frac{1}{2}$}
 \psfrag{C}{\footnotesize $1$}
    \begin{center}
     \includegraphics[clip,width=40mm]{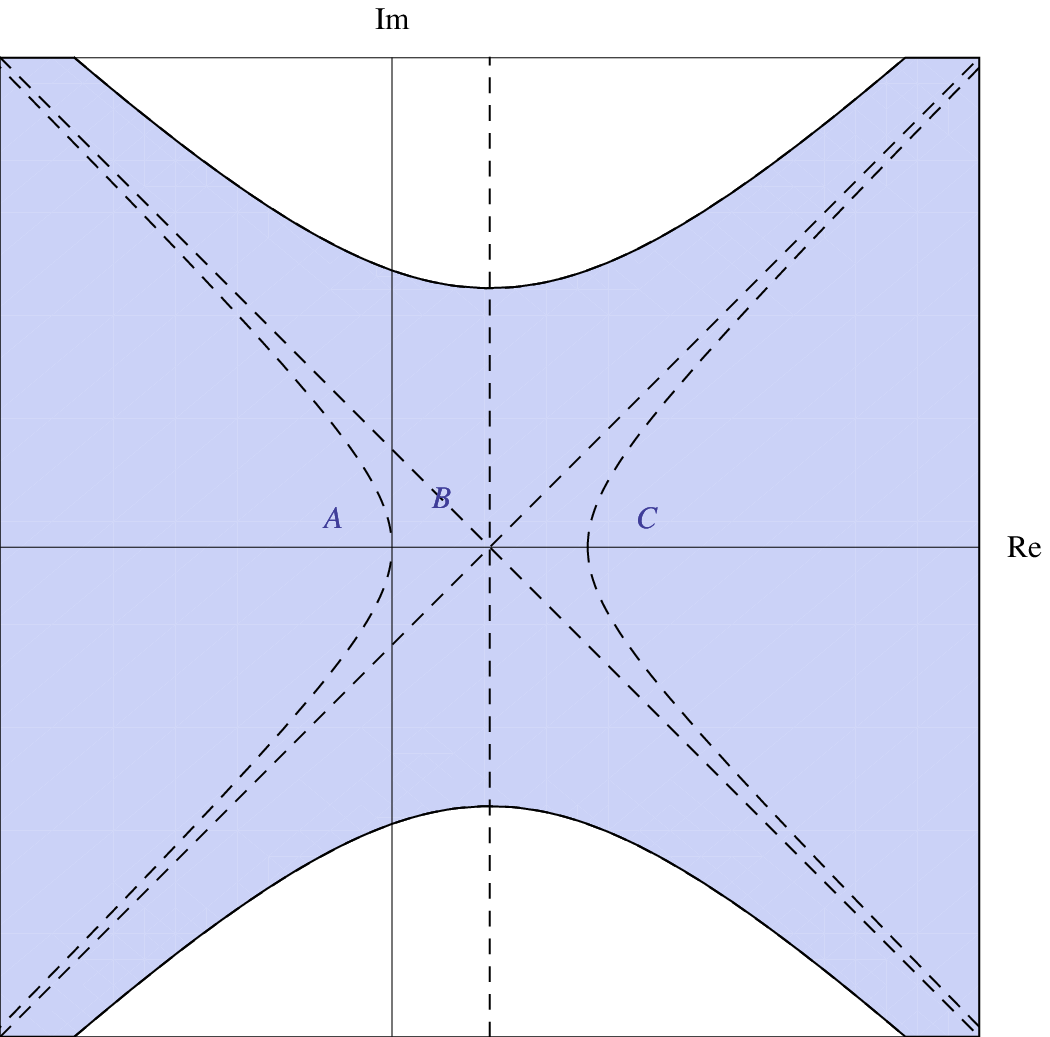}
     \caption{$U^{(T)}$ with $T>\frac{1}{4}$.}
    \end{center}
   \end{minipage}
  \end{tabular}
\end{center}
\end{figure}


\begin{remark}
\label{rem:strict}
 Define the region $U^{\pm}:=U\cap\{s\in\bC\,|\,\pm\Re(s)>\pm\frac{1}{2}\}$.
 Notice that $U=U^{+}\sqcup U^{-}$.
 Since the function $\mathrm{D}_{\G,r}(s)$ is defined on each region $U^{\pm}$, 
 let us write the restriction of $\mathrm{D}_{\G,r}$ to $U^{\pm}$ 
 as $\mathrm{D}^{\pm}_{\G,r}$.
 It is clear that
\begin{equation}
\label{for:fepm}
 \mathrm{D}^{\mp}_{\G,r}(1-s)=\mathrm{D}^{\pm}_{\G,r}(s) \qquad (s\in U^{\pm}).
\end{equation}
 When $r\ge 2$ and $\frac{1}{4}\notin\Spec(\Delta_{\G})$,
 the statements of Theorem~\ref{thm:HDDs} are equivalent to the following.
 Both functions $\mathrm{D}^{\pm}_{\G,r}(s)$
 can be continued analytically to the region $\Omega_{\G}$ and 
 the equation $\mathrm{D}^{+}_{\G,r}(s)=\mathrm{D}^{-}_{\G,r}(s)(=\mathrm{D}_{\G,r}(s))$
 ($s\in\Omega_{\G}$) follows from the identity theorem.
 Hence, from  \eqref{for:fepm},
 the functional equation of $\mathrm{D}_{\G,r}(s)$ is immediate. 
\end{remark}

\begin{remark}
\label{rem:branch}
 Let $r\ge 2$.
 Then, the function $\mathrm{D}_{\G,r}(s)$ is also extended beyond the region $U$
 even if we take the principal branch of the logarithm 
 to define $(\l_j-s(1-s))^{-w}$ for all $j\in\bN_{0}$.
 Actually, by the same manner as above,
 one can show that $\mathrm{D}_{\G,r}(s)$ admits an analytic continuation to the region 
 $\bC\setminus ([0,1]\cup(\frac{1}{2}+i\bR))$,
 which never become connected for any $\G$.  
 Moreover, one finds that 
 our choice of the log-branch is compatible with
 a functional equation and an Euler product expression of
 the Milnor-Selberg zeta function $Z_{\G,r}(s)$ established in Section~\ref{sec:zeta factor}. 
\end{remark}

\subsection{Calculations of the spectral zeta function}

 In this subsection,
 we study the function $\z_{\D_{\G}}(w,-s(1-s))$ by using the Selberg trace formula
\begin{align}
\label{for:STF}
 \sum^{\infty}_{j=0}\hat{f}(r_j)
=\sum_{\g\in\Hyp(\G)}
\frac{\log N(\d_\g)}{N(\g)^{\frac{1}{2}}-N(\g)^{-\frac{1}{2}}}f\bigl(\log N(\g)\bigr)
+(g-1)\int^{\infty}_{-\infty}\hat{f}(r)r\tanh(\pi r)dr.
\end{align}
 Here, $\Hyp(\G)$ is the set of all hyperbolic conjugacy classes in $\G$,
 ${\d}_\g\in\Prim(\G)$ for $\g\in\Hyp(\G)$ is the unique element 
 satisfying $\g={\d}_\g^k$ for some $k\ge 1$ and   
 $f$ is a test function whose Fourier transform
 $\hat{f}(r):=\int^{\infty}_{-\infty}f(x)e^{-irx}dx$ 
 satisfies the conditions 
 $\hat{f}$ is holomorphic in the band $\{r\in\bC\,|\,|\Im r|<\frac{1}{2}+\d\}$, 
 $\hat{f}(-r)=\hat{f}(r)$
 and $\hat{f}(r)=O(|r|^{-2-\d})$ as $|r|\to\infty$ for some $\d>0$,  

 Let $\Re(w)>r$ with $r\in\bN$ and $s\in U$.
 Throughout the present paper, we always denote by $t=t(s):=s-\frac{1}{2}$.
 From \eqref{for:intspectral} with $T=0$, we have 
\[
 \z_{\D_{\G}}\bigl(w+1-r,-s(1-s)\bigr)
=\frac{1}{\G(w+1-r)}\int^{\infty}_{0}\xi^{w+1-r}e^{-t^2\xi}
\Bigl(\sum^{\infty}_{j=0}e^{-r_j^2\xi}\Bigr)\frac{d\xi}{\xi}.
\]
 Applying the trace formula \eqref{for:STF} with the test function
 $f(x)=\frac{1}{2\sqrt{\pi\xi}}e^{-\frac{x^2}{4\xi}}$
 (then $\hat{f}(r)=e^{-r^2\xi}$)
 to the inner sum in the above integral
 and changing the order of the integrations, we have
\begin{align}
\label{for:sz}
 \z_{\D_{\G}}\bigl(w+1-r,-s(1-s)\bigr)
&=\Theta_{\G,r}(w,t)
+(g-1)\int^{\infty}_{-\infty}\bigl(x^2+t^2\bigr)^{-w+r-1}x\tanh(\pi x)dx\\
&=\Theta_{\G,r}(w,t)
+2(g-1)\sum^{r-1}_{\ell=0}\binom{r-1}{\ell}t^{2(r-1-\ell)}J_{2\ell+1}(w,t),
\nonumber
\end{align}
 where
\[
 \Theta_{\G,r}(w,t):
=\frac{1}{\G(w+1-r)}\int^{\infty}_{0}\xi^{w+1-r}\theta_{\G}(\xi,t)\frac{d\xi}{\xi}
\]
 with 
\begin{equation*}
 \theta_{\G}(\xi,t):
=\frac{1}{2\sqrt{\pi\xi}}
\sum_{\g\in\Hyp(\G)}
\frac{\log N(\d_\g)}{N(\g)^{\frac{1}{2}}-N(\g)^{-\frac{1}{2}}}
e^{-t^2\xi-\frac{(\log N(\g))^2}{4\xi}}
\end{equation*}
 and 
\begin{equation}
\label{def:J}
 J_{m}(w,t)
:=\int^{\infty}_{0}\vp_m(x;w,t)dx
\end{equation}
 with
\[
 \vp_m(x;w,t):=(x^2+t^2)^{-w}x^{m}\tanh(\pi x).
\] 
 Here, $(x^2+t^2)^{-w}:=\exp(-w\log{(x^2+t^2)})$.
 We call $J_{m}(w,t)$ an $m$-th moment function. 
 When $s\in U$, since $\Re(t^2)>\frac{1}{4}$ and 
 there exists a constant $\e_{\G}>0$ such that $\log{N(\g)}>\e_{\G}$ for all $\g\in\Hyp(\G)$,
 the integral in $\Theta_{\G,r}(w,t)$ converges absolutely for all $w\in\bC$
 and hence $\Theta_{\G,r}(w,t)$ defines an entire function as a function of $w$. 
 Moreover, we will show in the next section that
 $J_{m}(w,t)$ can be continued meromorphically to the whole $w$-plane $\bC$
 and is in particular holomorphic at $w=0$.
 Therefore, differentiating both sides of \eqref{for:sz} at $w=0$,
 we see that the higher depth determinant
 $\mathrm{D}_{\G,r}(s)$ 
 can be expressed as   
\begin{equation}
\label{for:det_r2}
 \mathrm{D}_{\G,r}(s)=G_{\G,r}(s)Z_{\G,r}(s),
\end{equation} 
 where $G_{\G,r}(s):=\phi_{r}(s)^{g-1}$ with 
\begin{align}
\label{def:gammafactor}
 \phi_{r}(s):
=\prod^{r-1}_{\ell=0}
\exp\Bigl(-\frac{\p}{\p w}J_{2\ell+1}(w,t)\Bigl|_{w=0}\Bigr)^{2\binom{r-1}{\ell}t^{2(r-1-\ell)}}
\end{align}
 and 
\begin{align}
\label{def:zeta}
 Z_{\G,r}(s):
&=\exp\Bigl(-\frac{\p}{\p w}\Theta_{\G,r}(w,t)\Bigl|_{w=0}\Bigr).
\end{align}
 Our next task is to examine each factor $\phi_r(s)$ and $Z_{\G,r}(s)$ more precisely.

\section{Gamma factors $\phi_r(s)$}
\label{sec:gamma factor}

\subsection{Derivative of the moment function $J_{m}(w,t)$ at $w=0$}
\label{subsubsec:J}

 To obtain an explicit expression of the gamma factor $\phi_r(s)$,
 let us study the function $J_{m}(w,t)$ defined by \eqref{def:J}.
 Notice that, for a fixed $s\in\bC$,
 the integral $J_{m}(w,t)$ converges absolutely for $\Re(w)>\frac{m+1}{2}$
 and hence, as a function of $w$, defines a holomorphic function in the region.
 In what follows, from the definition \eqref{def:gammafactor}, 
 we may assume that $m$ is odd. 
 Moreover, for simplicity,
 we assume that $\frac{1}{2}<s<1$, which implies that $0<t<\frac{1}{2}$.
 Then, as a function of $z$,
 $\vp_m(z;w,t)$ is holomorphic in $\{z\in\bC\,|\,\Re(z)>0\}$ and 
 is naturally continued meromorphically 
 to the region $\bC\setminus ((-\infty,0]\cup [-it,it])$
 by extending $\log{(z^2+t^2)}$.

 Suppose that $\Re(w)>\frac{m+1}{2}$.
 Let $R$ be a sufficiently large positive integer 
 and $\e$ a small real number.
 Consider the counterclockwise contour integral
 with the integrand $\vp_m(z;w,t)$
 (on the $z$-plane) along the path
\[
 \mathcal{C}(t;R,\e):=
 [+0,R]\sqcup
 C^{+}(0;R)\sqcup
 [-R,-0]\sqcup
 L_{-}(t;\e)\sqcup
 C(it;\e)\sqcup
 L_+(t;\e),
\] 
 where $C^{+}(0;R)$ is the semi-circle in the upper half plane with radius $R$
 centered at the origin,
 $C(it;\e)$ is the circle with radius $\e$ centered at $it$ and 
 $L_{+}(t;\e)$ (resp. $L_{-}(t;\e)$) is the right (resp. left) side of the segment
 connecting $0$ and $(1-\e)it$.
 We take $\e$ so small that the circle $C(it;\e)$ contains no points of the form
 $i(k+\frac{1}{2})$, that is, the poles of $\tanh(\pi z)$, for all $k\in\bN_{0}$
 (see Figure~$7$). 

\begin{figure}[htbp]
 \label{fig:Contour}
\begin{center}
 \centering
 \psfrag{Re}{\small $\Re$}
 \psfrag{Im}{\small $\Im$}
 \psfrag{A}{\small $-R$}
 \psfrag{B}{\small $R$}
 \psfrag{C}{\small $C$}
 \psfrag{D}{\small $C^{+}$}
 \psfrag{R}{\small $iR$}
 \psfrag{X}{\small $L_{+}$}
 \psfrag{Y}{\small $L_{-}$}
 \psfrag{c}{\small $i\frac{1}{2}$}
 \psfrag{d}{\small \ $\vdots$}
 \psfrag{e}{\small $i(R-\frac{1}{2})$}
 \psfrag{x}{\small $it$}
  \includegraphics[clip,width=80mm]{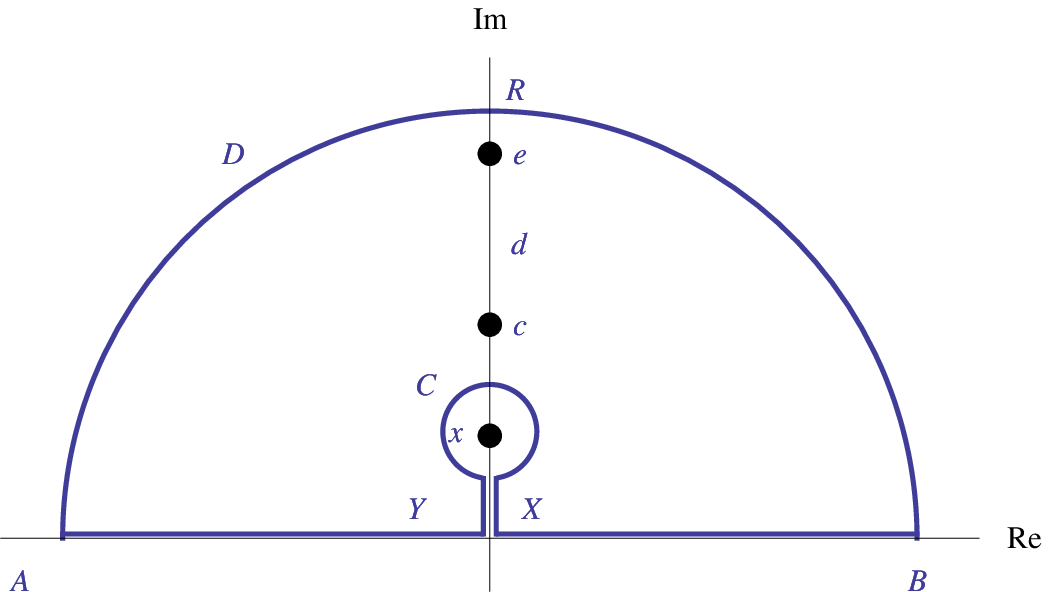}
 \caption{$\mathcal{C}(t;R,\e)$.}  
\end{center}
\end{figure}

 The residue theorem yields 
\begin{equation}
\label{for:int_all}
 \int_{\baC(t;R,\e)}\vp_m(z;w,t)dz
=2i^{m+1}e^{-\pi iw}\sum^{R-1}_{k=0}\bigl(k+\frac{1}{2}\bigr)^{-2w+m}
\Bigl(1-t^2\bigl(k+\frac{1}{2}\bigr)^{-2}\Bigr)^{-w}.
\end{equation}
 On the other hand, we have
\begin{align}
\label{for:int_all_another}
 \int_{\baC(t;R,\e)}\vp_m(z;w,t)dz 
&=\int_{C^{+}(0;R)}\vp_m(z;w,t)dz+\Biggl(\int_{[-R,-0]}+\int_{[+0,R]}\Biggr)\vp_m(z;w,t)dz\\
&\ \ \ +A_{m}(w,t;\e)+B_{m}(w,t;\e),\nonumber
\end{align}
 where
\begin{align*}
 A_{m}(w,t;\e)
&:=\Biggl(\int_{L_{+}(t;\e)}+\int_{L_{-}(t;\e)}\Biggr)\vp_m(z;w,t)dz, \\ 
 B_{m}(w,t;\e)
&:=\int_{C(it;\e)}\vp_m(z;w,t)dz.
\end{align*}
 Now, let us calculate each integral in the righthand-side of \eqref{for:int_all_another}.
 At first, it is easy to see that
 the integral on $C^{+}(0;R)$ converges to $0$ as $R\to+\infty$
 because it is $O(R^{m+1-2\Re(w)})$.
 Next, since $m$ is odd and
 $\arg(z^2+t^2)=0$ (resp. $2\pi$) if $z\in [+0,R]$ (resp. $z\in [-R,-0]$),
 we have  
\begin{align*}
 \Biggl(\int_{[-R,-0]}+\int_{[+0,R]}\Biggr)\vp_m(z;w,t)dz
&=(1+e^{-2\pi iw})\int^{R}_{0}\vp_m(x;w,t)dx.\\
&=2e^{-\pi iw}\cos(\pi w)\int^{R}_{0}\vp_m(x;w,t)dx.
\end{align*}
 This shows that  
\begin{align*}
 \lim_{R\to+\infty}\Biggl(\int_{[-R,-0]}+\int_{[+0,R]}\Biggr)\vp_m(z;w,t)dz
=2e^{-\pi iw}\cos(\pi w)J_{m}(w,t).
\end{align*}
 Finally, we calculate $A_{m}(w,t;\e)$ and $B_{m}(w,t;\e)$.
 Notice that both functions are entire as functions of $w$
 and, by the Cauchy integral theorem,
 do not depend on the choice of $\e$.
 Since $\arg(z^2+t^2)=0$ (resp. $2\pi$)
 if $z\in L_{+}(t;\e)$ (resp. $z\in L_{-}(t;\e)$), we have
\begin{align*}
 A_{m}(w,t;\e)
&=(1-e^{-2\pi iw})\int^{(1-\e)t}_{0}\varphi_m(i\xi;w,t)id\xi\\
&=2i^{m+1}t^{-2w}e^{-\pi iw}\sin(\pi w)
\int^{(1-\e)t}_{0}\bigl(1-\frac{\xi^2}{t^2}\bigr)^{-w}\xi^m\tan(\pi \xi)d\xi,
\end{align*} 
 whence
\begin{align*}
 A_{m}(w,t;\e)=\lim_{\e\to 0}A_{m}(w,t;\e)
&=2i^{m+1}t^{-2w}e^{-\pi iw}\sin(\pi w)
\int^{t}_{0}\bigl(1-\frac{\xi^2}{t^2}\bigr)^{-w}\xi^m\tan(\pi \xi)d\xi.
\end{align*} 
 Moreover, by straightforward calculations, 
 we see that $B_{m}(w,t;\e)=O(\e^{1-\Re(w)})$ as $\e\to 0$.
 This shows that $B_{m}(w,t;\e)=\lim_{\e\to 0}B_{m}(w,t;\e)=0$ for $\Re(w)<1$,
 whence, by the identity theorem, $B_{m}(w,t;\e)=0$ for all $w\in\bC$.
 Therefore, letting $R\to+\infty$
 in the formulas \eqref{for:int_all} and \eqref{for:int_all_another},
 we have 
\begin{align}
\label{for:Jn_middle}
 J_{m}(w,t)
&=\frac{i^{m+1}}{\cos(\pi w)}
\Bigl(J_{m,1}(w,t)+J_{m,2}(w,t)\Bigr),
\end{align}
 where 
\begin{align}
\label{for:Jm1}
 J_{m,1}(w,t)
&:=-t^{-2w}\sin(\pi w)
\int^{t}_{0}\bigl(1-\frac{\xi^2}{t^2}\bigr)^{-w}\xi^m\tan(\pi \xi)d\xi,\\
\label{for:Jm2}
 J_{m,2}(w,t)
&:=\sum^{\infty}_{k=0}\bigl(k+\frac{1}{2}\bigr)^{-2w+m}
\Bigl(1-t^2\bigl(k+\frac{1}{2}\bigr)^{-2}\Bigr)^{-w}.
\end{align}

 The function $J_{m,1}(w,t)$ is clearly entire as a function of $w$.
 Let us calculate the derivative of $J_{m,1}(w,t)$ at $w=0$.
 To do that, we employ the following basic multiple trigonometric functions
 (\cite{KOW1}, see also \cite{KK1}).
 Put $P_1(u):=(1-u)$ and
 $P_n(u):=(1-u)\exp(u+\frac{u^2}{2}+\cdots+\frac{u^n}{n})$ for $n\ge 2$. 
 Then, the basic multiple sine $\baS_n(z)$ and
 cosine function $\baC_n(z)$ are respectively defined by 
\begin{align*}
 \baS_n(z):
&=
\begin{cases}
\DS{2\pi z
\prod_{m\in\bZ\atop m\ne 0}
P_1\bigl(\frac{z}{m}\bigr)
=2\pi z\prod^{\infty}_{m=1}\Bigl(1-\frac{z^2}{m^2}\Bigr)
=2\sin(\pi z)}  & (n= 1),\\
\DS{\exp\Bigl(\frac{z^{n-1}}{n-1}\Bigr)
\prod_{m\in\bZ\atop m\ne 0}
\Bigl(P_n\bigl(\frac{z}{m}\bigr)P_n\bigl(-\frac{z}{m}\bigr)^{(-1)^{n-1}}\Bigr)^{m^{n-1}}}
 & (n\ge 2),
\end{cases}\\
 \baC_n(z):
&=\prod_{m\in\bZ\atop m:\mathrm{odd}}
 P_n\Bigl(\frac{z}{(\frac{m}{2})}\Bigr)^{(\frac{m}{2})^{n-1}}
=\baS_n(2z)^{2^{1-n}}\baS_n(z)^{-1}.
\end{align*}
 Notice that, when $n\ge 2$, $\baS_n(0)=\baC_n(0)=1$.
 Moreover, for $z\ne 0$, 
 we have the following integral expressions. 
\begin{align}
\label{for:int_multtrig}
 \baS_n(z)
=\exp\Bigl(\int^{z}_{0}\pi \xi^{n-1}\cot(\pi \xi)d\xi\Bigr),\quad 
 \baC_n(z)
=\exp\Bigl(-\int^{z}_{0}\pi \xi^{n-1}\tan(\pi \xi)d\xi\Bigr).
\end{align}

\begin{prop} 
 It holds that   
\begin{equation}
\label{for:der_J1}
 \frac{\p}{\p w}J_{m,1}(w,t)\Bigl|_{w=0}
=\log{\baC_{m+1}(t)}.
\end{equation}
\end{prop}
\begin{proof}
 From \eqref{for:Jm1}, we have
\begin{align*}
 J_{m,1}(w,t)
&=-\bigl(1+O(w)\bigr)\bigl(\pi w+O(w^3)\bigr)
\int^{t}_{0}\bigl(1+O(w)\bigr)\xi^{m}\tan(\pi \xi)d\xi\\
&=\Bigl(-\int^{t}_{0}\pi \xi^{m}\tan(\pi \xi)d\xi\Bigr)w+O(w^2).
\end{align*}
 Hence, the claim follows from the formula \eqref{for:int_multtrig}.
\end{proof}

 We next study the function $J_{m,2}(w,t)$. 
 We first show the following 

\begin{lem}
 We have   
\begin{align}
\label{for:J2}
 J_{m,2}(w,t)
&=\sum^{\infty}_{j=0}\binom{w+j-1}{j}
t^{2j}\z\bigl(2w+2j-m,\frac{1}{2}\bigr).
\end{align}
 This gives a meromorphic continuation to the whole plane $\bC$
 as a function of $w$ with possible simple poles at $w=\frac{m+1}{2}-j$
 for $j\in\bN_{0}$. In particular, it is holomorphic at $w=0$.
 \end{lem}
\begin{proof}
 Notice that,
 since $0<t<\frac{1}{2}$ (recall that we assume that $\frac{1}{2}<s<1$),
 it holds that $|t^2(k+\frac{1}{2})^{-2}|<1$ for all $k\ge 0$.
 Therefore, from \eqref{for:Jm2}, using the binomial theorem, we have
\begin{align*}
 J_{m,2}(w,t)
&=\sum^{\infty}_{k=0}\bigl(k+\frac{1}{2}\bigr)^{-2w+m}
\sum^{\infty}_{j=0}\binom{w+j-1}{j}t^{2j}\bigl(k+\frac{1}{2}\bigr)^{-2j}\\
&=\sum^{\infty}_{j=0}\binom{w+j-1}{j}t^{2j}\z\bigl(2w+2j-m,\frac{1}{2}\bigr).
\end{align*}
 Hence one obtains the expression \eqref{for:J2}. 
 Moreover, since $\z(2w+2j-m,\frac{1}{2})$ is uniformly bounded with respect to $j$,
 this gives a meromorphic continuation to $\bC$. 
 Now the rest of assertions is clear
 because the Hurwitz zeta function $\z(w,z)$ has a simple pole at $w=1$.
 Notice that the origin is not a pole of $J_{m,2}(w,t)$ 
 since, when $j=\frac{m+1}{2}\ge 1$, $\binom{w+j-1}{j}=O(w)$ as $w\to 0$.  
\end{proof}

 Before calculating the derivative of $J_{m,2}(w,t)$ at $w=0$,
 let us recall the multiple gamma functions,
 which is defined via the Barnes multiple zeta function (\cite{Barnes}) 
\[
 \z_n(w,z)
:=\sum_{m_1,\ldots,m_n\ge 0}\frac{1}{(m_1+\cdots+m_n+z)^w}
\qquad (\Re(w)>n).
\]
 This clearly gives a generalization of the Hurwitz zeta function; $\z_{1}(w,z)=\z(w,z)$.
 It is known that $\z_n(w,z)$ can be continued meromorphically to the whole plane $\bC$
 with possible simple poles at $w=1,2,\ldots,n$.
 The multiple gamma function $\G_{n,r}(z)$ of depth $r$ is defined by
\[
 \G_{n,r}(z)
:=\exp\Bigl(\frac{\p}{\p w}\z_{n}(w,z)\Bigl|_{w=1-r}\Bigr).
\]
 In particular, we put $\G_{n}(z):=\G_{n,1}(z)$ and $\mG_{r}(z):=\G_{1,r}(z)$.
 These are respectably called the Barnes multiple gamma function (\cite{Barnes}) and
 the Milnor gamma function of depth $r$ (\cite{M}, see also \cite{KOW2}).
 From the Lerch formula 
 $\frac{\p}{\p w}\z(w,z)\bigl|_{w=0}=\log\frac{\G(z)}{\sqrt{2\pi}}$, we have  
 $\G_{1,1}(z)=\G_{1}(z)=\mG_1(z)=\frac{\G(z)}{\sqrt{2\pi}}$,
 whence these in fact give generalizations of the classical gamma function. 
 We remark that $\G_n(z)^{-1}$ is an entire function
 with zeros at $z=-k$ of order $\binom{k+n-1}{n-1}$ for $k\in\bN_{0}$.

 From the expression \eqref{for:J2},
 it can be written as  
\begin{align*}
 J_{m,2}(w,t)
&=\z\bigl(2w-m,\frac{1}{2}\bigr)
+\sum^{\frac{m-1}{2}}_{j=1}\binom{w+j-1}{j}
t^{2j}\z\bigl(2w+2j-m,\frac{1}{2}\bigr)\\
&\ \ \ +\binom{w+\frac{m+1}{2}-1}{\frac{m+1}{2}}
t^{m+1}\z\bigl(2w+1,\frac{1}{2}\bigr)
+\sum^{\infty}_{j=\frac{m+3}{2}}\binom{w+j-1}{j}
t^{2j}\z\bigl(2w+2j-m,\frac{1}{2}\bigr)\\
&=:T_1(w,t)+T_2(w,t)+T_3(w,t)+T_4(w,t).
\end{align*}
 Using the expansions
 $\binom{w+j-1}{j}=\frac{1}{j}(w+H(j-1)w^2+O(w^3))$ as $w\to 0$ for $j\ge 1$
 where $H(m):=\sum^{m}_{k=1}\frac{1}{k}$ and 
 $\z(w,z)=\frac{1}{w-1}-\psi(z)+O(w-1)$ as $w\to 1$
 where $\psi(z):=\frac{d}{dz}\log{\G(z)}=\frac{\G'}{\G}(z)$ is the digamma function and
 employing the formula $\z(1-m,z)=-\frac{B_m(z)}{m}$ for $m\in\bN$
 where $B_m(z)$ is the Bernoulli polynomial defined by 
 $\frac{xe^{zx}}{e^x-1}=\sum^{\infty}_{m=0}B_m(z)\frac{x^m}{m!}$, 
 we have 
\begin{align}
\label{for:T1}
 T_1(w,t)
&=\z\bigl(-m,\frac{1}{2}\bigr)
+\frac{\p}{\p w}\z\bigl(2w-m,\frac{1}{2}\bigr)\Bigl|_{w=0}w
+O(w^2)\\
&=-\frac{B_{m+1}(\frac{1}{2})}{m+1}
+\Bigl(2\log{\mG_{m+1}\bigl(\frac{1}{2}\bigr)}\Bigr)w+O(w^2),\nonumber\\
\label{for:T2}
 T_2(w,t)
&=\sum^{\frac{m-1}{2}}_{j=1}
\frac{1}{j}\bigl(w+O(w^2)\bigr)t^{2j}
\Bigl(\z\bigl(2j-m,\frac{1}{2}\bigr)+O(w)\Bigr)\\
&=\Bigl(-\sum^{\frac{m-1}{2}}_{j=1}
\frac{B_{m+1-2j}(\frac{1}{2})}{j(m+1-2j)}t^{2j}\Bigr) w+O(w^2),\nonumber\\
\label{for:T3}
 T_3(w,t)
&=\frac{t^{m+1}}{\frac{m+1}{2}}
\Bigl(w+H\bigl(\frac{m-1}{2}\bigr)w^2+O(w^3)\Bigr)
\Bigl(\frac{1}{2w}-\psi\bigl(\frac{1}{2}\bigr)+O(w^2)\Bigr)\\
&=\frac{1}{m+1}t^{m+1}
+\biggl(\frac{2}{m+1}
\Bigl(
\frac{1}{2}H\bigl(\frac{m-1}{2}\bigr)-\psi\bigl(\frac{1}{2}\bigr)
\Bigr)t^{m+1}\biggr)w+O(w^2),\nonumber\\
\label{for:T4}
 T_4(w,t)
&=\sum^{\infty}_{j=\frac{m+3}{2}}
\frac{1}{j}\bigl(w+O(w^2)\bigr)t^{2j}
\Bigl(\z\bigl(2j-m,\frac{1}{2}\bigr)+O(w)\Bigr)\\
&=\bigl(2R_m(t)\bigr)w+O(w^2)\nonumber
\end{align}
 as $w\to 0$.
 Here, $R_m(t)$ is defined by 
\[
 R_m(t):
=\sum^{\infty}_{j=1}\frac{\z\bigl(2j+1,\frac{1}{2}\bigr)}{2j+m+1}t^{2j+m+1}.
\] 
 This yields the following     

\begin{prop}
 It holds that
\begin{align}
\label{for:der_J2}
 \frac{\p}{\p w}J_{m,2}(w,t)\Bigl|_{w=0}
&=2\sum^{m+1}_{k=1}(-1)^{k}\binom{m}{k-1}t^{m+1-k}
\log{\mG_{k}\bigl(t+\frac{1}{2}\bigr)}\\
&\ \ \ 
-\log{\baC_{m+1}(t)}
+\frac{1}{m+1}\Bigl(H\bigl(\frac{m-1}{2}\bigr)-2H(m)\Bigr)t^{m+1}.\nonumber
\end{align}
\end{prop}
\begin{proof}
 This follows from the identity
 $\frac{\p}{\p w}J_{m,2}(w,t)\bigl|_{w=0}=\sum^{4}_{j=1}\frac{\p}{\p w}T_{j}(w,t)\bigl|_{w=0}$
 together with \eqref{for:T1}, \eqref{for:T2}, \eqref{for:T3}, \eqref{for:T4} and 
 the following proposition, which will be proved in Subsection~\ref{ssec:series}.
\end{proof}

\begin{prop}
\label{prop:Pa}
 Let $m\ge 1$ be an odd integer. Then, we have 
\begin{align}
\label{for:Rm}
 R_m(t)
&=\sum^{m+1}_{k=1}(-1)^{k}\binom{m}{k-1}t^{m+1-k}
\log{\mG_{k}\bigl(t+\frac{1}{2}\bigr)}
-\frac{1}{2}\log{\baC_{m+1}(t)}\\
&\ \ \ 
-\frac{1}{m+1}\Bigl(H(m)-\psi\bigl(\frac{1}{2}\bigr)\Bigr)t^{m+1}
+\frac{1}{2}\sum^{\frac{m-1}{2}}_{j=1}
\frac{B_{m+1-2j}(\frac{1}{2})}{j(m+1-2j)}t^{2j}
-\log{\mG_{m+1}\bigl(\frac{1}{2}\bigr)}.\nonumber
\end{align}
\end{prop}

 From the equation \eqref{for:Jn_middle},
 noting that $\cos(\pi w)^{-1}=1+O(w^2)$ as $w\to 0$,
 we have
\[
 \frac{\p}{\p w}J_{m}(w,t)\Bigl|_{w=0}
=i^{m+1}\biggl(\frac{\p}{\p w}J_{m,1}(w,t)\Bigr|_{w=0}+\frac{\p}{\p w}J_{m,2}(w,t)\Bigl|_{w=0}\biggr).
\] 
 Therefore, substituting \eqref{for:der_J1} and \eqref{for:der_J2} into this equation,
 one eventually obtains the derivative of the moment function $J_{m}(w,t)$ at $w=0$.

\begin{prop}
 It holds that 
\begin{align}
\label{for:der_Jm_middle}
 \frac{\p}{\p w}J_{m}(w,t)\Bigl|_{w=0}
&=\frac{i^{m+1}}{m+1}\Bigl(H\bigl(\frac{m-1}{2}\bigr)-2H(m)\Bigr)t^{m+1}\\
&\ \ \ 
 +2i^{m+1}\sum^{m+1}_{k=1}(-1)^k\binom{m}{k-1}t^{m+1-k}\log{\mG_{k}\bigl(t+\frac{1}{2}\bigr)}.\nonumber
\end{align}
\qed
\end{prop}

\subsection{Explicit expressions of $\phi_{r}(s)$}

 We obtain the following expression of $\phi_{r}(s)$
 by the Milnor gamma functions.

\begin{prop}
\label{prop:gammafactor}
 We have 
\begin{align}
\label{for:gammafactor}
 \phi_{r}(s)
&=e^{-\frac{(2r)!!}{r^2(2r-1)!!}(s-\frac{1}{2})^{2r}}
\prod^{2r}_{k=r}\mG_k(s)^{\binom{r}{k-r}\frac{2k}{r}(-1)^{k+r-1}(2s-1)^{2r-k}}.
\end{align}
\end{prop}

 To prove this,
 we need the following lemmas about sums of the binomial coefficients.

\begin{lem}
\label{lem:Cr}
 Let $r\in\bN$. Then, the following equality holds; 
\begin{align}
\label{for:Cr}
 C_r:
&=\sum^{r-1}_{\ell=0}\binom{r-1}{\ell}\frac{(-1)^{\ell}}{\ell+1}
\bigl(H(\ell)-2H(2\ell+1)\bigr)=-\frac{(2r)!!}{r^2(2r-1)!!}.
\end{align}
\end{lem}
\begin{proof}
 Since $H(\ell)-2H(2\ell+1)=-2\sum^{\ell}_{k=0}\frac{1}{2k+1}$
 (we understand that $H(\ell)=0$ if $\ell\le 0$),
 changing the order of the summations, we have
\begin{equation}
\label{for:midC}
 C_r
=-2\sum^{r-1}_{k=0}\frac{1}{2k+1}
\sum^{r-1}_{\ell=k}\binom{r-1}{\ell}\frac{(-1)^{\ell}}{{\ell}+1}
=-\frac{2}{r}\sum^{r-1}_{k=0}\binom{r-1}{k}\frac{(-1)^k}{2k+1}.
\end{equation}
 Here, we have used the formula 
 $\sum^{r-1}_{\ell=k}\binom{r-1}{\ell}\frac{(-1)^{\ell}}{\ell+1}=\frac{(-1)^k}{r}\binom{r-1}{k}$,
 which is easily obtained by induction on $k$.
 Therefore, since the sum on the rightmost-hand side of \eqref{for:midC} is equal to 
 the beta integral
\[
 \int^{1}_{0}(1-x^2)^{r-1}dx
=\frac{1}{2}B(\frac{1}{2},r)=\frac{\G(\frac{1}{2})\G(r)}{2\G(r+\frac{1}{2})}
=\frac{(2r)!!}{2r(2r-1)!!},
\]
 one obtains the desired formula.
\end{proof}

\begin{lem}
\label{lem:Drk}
 $(\mathrm{i})$\ Let $r\in\bN$ and $1\le k\le 2r$. Then, the following equality holds; 
\begin{align}
\label{for:Drk}
 D_r(k):
=\sum^{r-1}_{\ell=\Gauss{\frac{k-1}{2}}}4(-1)^{\ell+k}\binom{r-1}{\ell}\binom{2\ell+1}{k-1}
=\binom{r}{k-r}\frac{2k}{r}(-1)^{k+r-1}2^{2r-k}.
\end{align}
 In particular, $D_r(k)=0$ if $1\le k\le r-1$.

 $(\mathrm{ii})$\ Let $1\le p\le 2r$. Then, the following equality holds; 
\begin{equation}
\label{for:DDD}
 \wt{D}_{r,p}:=
 \sum^{2r}_{k=p}\binom{k-1}{k-p}D_r(k)
=
\begin{cases}
 0 & (p:\textrm{odd}),\\[3pt]
 \DS{4\binom{r-1}{\frac{p}{2}-1}(-1)^{\frac{p}{2}-1}} & (p:\textrm{even}).
\end{cases} 
\end{equation}
\end{lem}
\begin{proof}
 From the binomial expansion
 $(1-t^2)^{r-1}=\sum^{r-1}_{\ell=0}\binom{r-1}{\ell}t^{2\ell}$,
 one sees that   
\begin{align}
\label{for:Drk-mid}
 D_r(k)
&=\frac{4(-1)^k}{(k-1)!}\frac{d^{k-1}}{dt^{k-1}}\Bigl(t\bigl(1-t^2\bigr)^{r-1}\Bigr)\Bigl|_{t=1}\\
&=\frac{4(-1)^k}{(k-1)!}\Biggl[\frac{d^{k-1}}{dt^{k-1}}\bigl(1-t^2\bigr)^{r-1}\Bigl|_{t=1}+(k-1)\frac{d^{k-2}}{dt^{k-2}}\bigl(1-t^2\bigr)^{r-1}\Bigl|_{t=1}\Biggr].\nonumber
\end{align}
 This shows that $D_r(k)=0$ if $1\le k\le r-1$.
 Now, let $r\le k\le 2r$ and $j\ge r-1$.
 Then, by the Leibniz rule, we have 
\begin{align*}
 \frac{d^{j}}{dt^{j}}\bigl(1-t^2\bigr)^{r-1}\Bigl|_{t=1}
&=\frac{d^{j}}{dt^{j}}\Bigl((1-t)^{r-1}(1+t)^{r-1}\Bigr)\Bigl|_{t=1}\\
&=\sum^{j}_{m=0}\binom{j}{m}\Bigl(\frac{d^{m}}{dt^{m}}(1-t)^{r-1}\Bigr)\Bigl|_{t=1}\cdot
\Bigl(\frac{d^{j-m}}{dt^{j-m}}(1+t)^{r-1}\Bigr)\Bigl|_{t=1}\\
&=j!\binom{r-1}{j-r+1}(-1)^{r-1}2^{2r-2-j}.
\end{align*}
 Hence, from \eqref{for:Drk-mid}, using this formula with $j=k-1$ and $k-2$, 
 one obtains \eqref{for:Drk}.
 Moreover, it is clear that the equation \eqref{for:Drk} is also valid for $1\le k\le r-1$
 because $\binom{r}{k-r}=0$ for such a $k$.
 Hence the claim $(\mathrm{i})$ follows. 
 We next show the claim $(\mathrm{ii})$. 
 Consider the generating function $\sum^{2r}_{p=1}\wt{D}_{r,p}x^p$.
 Changing the order of the summations and using the formula \eqref{for:Drk},
 we see that this is equal to 
\begin{align*}
 \sum^{2r}_{p=1}\sum^{2r}_{k=p}\binom{k-1}{k-p}D_r(k)x^p
&=\sum^{2r}_{k=1}D_r(k)x(1+x)^{k-1}\\
&=4x^2(1-x^2)^{r-1}\\
&=4\sum^{r}_{p=1}\binom{r-1}{p-1}(-1)^{p-1}x^{2p}.
\end{align*}
 Hence the claim follows.
\end{proof}

 We now give the proof of Proposition~\ref{prop:gammafactor}.

\begin{proof}
[Proof of Proposition~\ref{prop:gammafactor}]
 From the equation \eqref{for:der_Jm_middle},
 changing the order of the products, we have 
\begin{align}
\label{for:phiphi}
 \phi_{r}(s)
&=e^{C_rt^{2r}}\prod^{r-1}_{\ell=0}\prod^{2\ell+2}_{k=1}
\mG_{k}\bigl(t+\frac{1}{2}\bigr)^{4(-1)^{\ell+k}\binom{r-1}{\ell}\binom{2\ell+1}{k-1}t^{2r-k}}\\
&=e^{C_rt^{2r}}\prod^{2r}_{k=1}\mG_{k}(s)^{D_r(k)t^{2r-k}}.\nonumber
\end{align}
 Therefore, one immediately obtains the formula from \eqref{for:Cr} and \eqref{for:Drk}.
\end{proof}

 To clarify the analytic properties of $\phi_r(s)$,
 let us rewrite the expression \eqref{for:gammafactor}
 in terms of the Barnes multiple gamma functions.
 The following expression is obtained in \cite{KOW2};  
\begin{equation}
\label{for:reduce-mG}
 \mG_{r}(z)=\prod^{r}_{j=1}\G_{j}(z)^{c_{r,j}(z)},
\end{equation}
 where, for $r\ge 1$ and $j\ge 1$, 
 $c_{r,j}(z)$ is the polynomial in $z$ defined by 
\begin{align}
\label{for:crl}
 c_{r,j}(z)
&:=\sum^{j-1}_{l=0}\binom{j-1}{l}(-1)^{l}(z-l-1)^{r-1}.
\end{align}
 For example,
 $c_{r,r}(z)=(r-1)!$, $c_{r,r-1}(z)=\frac{1}{2}(2z-r)(r-1)!$,
 \ldots and $c_{r,1}(z)=(z-1)^{r-1}$.
 It is easy to see that the polynomial $c_{r,j}(z)$ satisfies the recursion formula
\[
 c_{r,j}(z)=(z-1)c_{r-1,j}(z)+(j-1)c_{r-1,j-1}(z-1).
\]
 From this, noting that $c_{1,j}(z)=(1-1)^{j-1}=0$ for $j\ge 2$,
 we have $c_{r,j}(z)=0$ for $1\le r\le j-1$.
 Notice that $c_{r,j}(z)$ is also given by the generating function
\begin{equation}
\label{for:genec}
 (T+z)^{r-1}=\sum^{r}_{j=1}c_{r,j}(z)\binom{T+j-1}{j-1}. 
\end{equation}

 For $x\in\bR$, let us denote $\Gauss{x}$ by the largest integer not exceeding $x$. 
 Then, we have the following expression of $\phi_r(s)$
 in terms of the Barnes multiple gamma functions.

\begin{thm}
\label{thm:gamma factor1}
 We have 
\begin{align}
\label{for:gammafactor1}
 \phi_{r}(s)
&=e^{-\frac{(2r)!!}{r^2(2r-1)!!}(s-\frac{1}{2})^{2r}}
\prod^{2r}_{j=1}\G_j(s)^{\a_{r,j}(s-\frac{1}{2})},
\end{align}
 where $\a_{r,j}(t)$ is the even polynomial defined by 
\begin{align}
\label{def:alpha}
 \a_{r,j}(t):
&=4\sum^{r}_{\ell=\Gauss{\frac{j+1}{2}}}
\binom{r-1}{\ell-1}(-1)^{\ell-1}c_{2\ell,j}
\bigl(\frac{1}{2}\bigr)t^{2r-2\ell}.
\end{align}
 This gives 
\begin{itemize}
 \item[$\mathrm{(i)}$] a meromorphic continuation of $\phi_r(s)$ to the whole plane $\bC$ 
 with poles at $s=-k$ of order $2(2k+1)$ for $k\in\bN_{0}$  when $r=1$.\ \\[-20pt]
 \item[$\mathrm{(ii)}$] an analytic continuation to the region
 $\bC\setminus \bigl((-\infty,-1]\cup [0,-i\infty)\bigr)$ when $r\ge 2$.
\end{itemize}
\end{thm}
\begin{proof}
 From the equations \eqref{for:phiphi} and \eqref{for:reduce-mG}, 
 it suffices to show that 
\begin{equation}
\label{for:suffice}
 \a_{r,j}(t)=\sum^{2r}_{k=j}c_{k,j}\bigl(t+\frac{1}{2}\bigr)D_r(k)t^{2r-k}.
\end{equation}
 Actually, from the equation \eqref{for:crl},
 the righthand-side of \eqref{for:suffice} is rewritten as 
\begin{align*}
&\ \ \ 
 \sum^{2r}_{k=j}\Biggl(\sum^{j-1}_{l=0}\binom{j-1}{l}(-1)^{l}\bigl(t-\frac{1}{2}-l\bigr)^{k-1}\Biggr)D_r(k)t^{2r-k}\\
&=\sum^{j-1}_{l=0}\binom{j-1}{l}(-1)^{l}\Biggl(\sum^{2r}_{k=j}\sum^{k-1}_{m=0}\binom{k-1}{m}\bigl(-\frac{1}{2}-l\bigr)^{(k-m)-1}D_r(k)t^{2r-(k-m)}\Biggr).
\end{align*}
 Putting $k-m=p$ in the inner sums, one sees that this is equal to
\begin{align*}
&\ \ \
 \sum^{2r}_{p=1}\Biggl(\sum^{2r}_{k=\max\{p,j\}}\binom{k-1}{k-p}D_r(k)\Biggr)\Biggl(\sum^{j-1}_{l=0}\binom{j-1}{l}(-1)^{l}\bigl(\frac{1}{2}-l-1\bigr)^{p-1}\Biggr)t^{2r-p}\\
&=\sum^{2r}_{p=1}\Biggl(\sum^{2r}_{k=\max\{p,j\}}\binom{k-1}{k-p}D_r(k)\Biggr)c_{p,j}\bigl(\frac{1}{2}\bigr)t^{2r-p}.
\end{align*}
 Moreover, the summand for $p=1,2,\ldots,j-1$ vanishes because $c_{p,j}(z)=0$,
 whence, consequently, this can be written as 
 $\sum^{2r}_{p=j}\wt{D}_{r,p}c_{p,j}(\frac{1}{2})t^{2r-p}$.
 Therefore, from \eqref{for:DDD}, one obtains the equation \eqref{for:suffice}.
 The rest of the assertion follows from the equations $\a_{1,1}(t)=-2$ and $\a_{1,2}(t)=4$
 and the fact that the point $z=-k$  is a zero of
 $\G_j(z)^{-1}$ of order $\binom{k+j-1}{j-1}$.
 This ends the proof.
\end{proof}

\begin{remark}
\label{rem:constant}
 From the equations $c_{r,r-1}(z)=\frac{1}{2}(2z-r)(r-1)!$ and $c_{r,r}(z)=(r-1)!$,
 one sees that both $\a_{r,2r-1}(t)$ and $\a_{r,2r}(t)$ are integers
 respectively given by 
\begin{align*}
 \a_{r,2r-1}(t)&
=4(-1)^{r-1}c_{2r,2r-1}\bigl(\frac{1}{2}\bigr)=(-1)^r(4r-2)(2r-1)!,\\
 \a_{r,2r}(t)&
=4(-1)^rc_{2r,2r}\bigl(\frac{1}{2}\bigr)=4(-1)^{r-1}(2r-1)!.
\end{align*}
\end{remark}

%
%
%

 We next give an expression of $\phi_{r}(s)$
 via the Vign\'eras multiple gamma function $G_n(z)$,
 which are characterized by a generalization of the Bohr-Mollerup theorem (\cite{Vigneras1979}). 
 Notice that $G_1(z)=\G(z)$ and $G_2(z)=G(z)$
 where $G(z)$ is the Barnes $G$-function studied in \cite{Barnes1899}. 
 From \cite[p.\,$87$ $(27)$]{SrivastavaChoi2001}, 
 we know that $G_n(z)$ is essentially equal to $\G_n(z)$;
\begin{equation}
\label{for:Vigneras-Barnes2}
 G_n(z)
=e^{(-1)^n\sum^{n-1}_{j=0}b_{n,j}(z)\z'(-j)}\cdot \G_n(z)^{(-1)^{n-1}}.
\end{equation}
 Here, $b_{n,j}(z)$ is the polynomial of degree $n-1-j$
 defined by the generating function $\binom{j+n-1}{n-1}=\sum^{n-1}_{k=0}b_{n,k}(z)(j+z)^{k}$.
 To be more precise, let $s(n,m)$ be the Stirling number of the first kind defined by
 $(z)_{n}=\sum^{n}_{m=0}(-1)^{n+m}s(n,m)z^m$
 where $(z)_n:=\frac{\G(z+n)}{\G(z)}=z(z+1)\cdots (z+n-1)$ is the Pochhammer symbol. 
 Then, it is given by 
 $b_{n,k}(z)=\frac{(-1)^{n-1-k}}{(n-1)!}\sum^{n-1}_{m=k}\binom{m}{k}s(n,m+1)z^{m-k}$.
 The following expression is immediately obtained from \eqref{for:gammafactor1}
 together with the identity \eqref{for:Vigneras-Barnes2}.

\begin{cor}
\label{cor:gamma factor2}
 We have 
\begin{align}
\label{for:gammafactor2}
 \phi_{r}(s)
&=e^{-\frac{(2r)!!}{r^2(2r-1)!!}(s-\frac{1}{2})^{2r}+\sum^{2r-1}_{\ell=0}
\b_{r,\ell}(s-\frac{1}{2})\z'(-\ell)}
\prod^{2r}_{j=1}G_j(s)^{(-1)^{j-1}\a_{r,j}(s-\frac{1}{2})}.
\end{align}
 Here, $\b_{r,\ell}(t)$ is the polynomial defined by  
\begin{align*}
 \b_{r,\ell}(t):
&=\sum^{2r}_{j=\ell+1}b_{j,\ell}\bigl(t+\frac{1}{2}\bigr)\a_{r,j}(t).
\end{align*} 
\qed
\end{cor}

\begin{example}
 Let $t=s-\frac{1}{2}$.
 Then, from \eqref{for:gammafactor1} and \eqref{for:gammafactor2}, we have  
\begin{align*}
 \phi_{1}(s)
&=e^{-2t^2}\G_1(s)^{-2}\G_2(s)^4\\
&=e^{-2t^2-4t\z'(0)+4\z'(-1)}G_1(s)^{-2}G_2(s)^{-4},\\
 \phi_{2}(s)
&=e^{-\frac{2}{3}t^4}
\G_1(s)^{-2t^2+\frac{1}{2}}\G_2(s)^{4t^2-13}\G_3(s)^{36}\G_4(s)^{-24}\\
&=e^{-\frac{2}{3}t^4-8t^2\z'(-1)+12t\z'(-2)-4\z'(-3)}
G_1(s)^{-2t^2+\frac{1}{2}}G_2(s)^{-4t^2+13}G_3(s)^{36}G_4(s)^{24},\\
 \phi_3(s)
&=e^{-\frac{16}{45}t^6}\G_1(s)^{-2t^4+t^2-\frac{1}{8}}\G_2(s)^{4t^4-26t^2+\frac{121}{4}}\G_3(s)^{72t^2-330}\G_4(s)^{-48t^2+1020}\G_5(s)^{-1200}\G_6(s)^{480}\\
&=e^{-\frac{16}{45}t^6-16t^3\z'(-2)+32t^2\z'(-3)-20t\z'(-4)+4\z'(-5)}\\
&\ \ \ \times G_1(s)^{-2t^4+t^2-\frac{1}{8}}
 G_2(s)^{-4t^4+26t^2-\frac{121}{4}}G_3(s)^{72t^2-330}G_4(s)^{48t^2-1020}G_5(s)^{-1200} G_6(s)^{-480}.
\end{align*}
 Notice that the expression of $\phi(s)=\phi_{1}(s)$ is obtained in \cite{V}.
\end{example}

\subsection{Functional equations of $\phi_{r}(s)$}
\label{ssec:fegamma}

 In this subsection, we establish a functional equation of $\phi_r(s)$.
 To do that, we first recall the normalized multiple sine function $S_n(z)$
 studied in \cite{KK1};
\begin{equation}
\label{def:multsine}
 S_n(z):=\G_n(z)^{-1}\G_n(n-z)^{(-1)^n}.
\end{equation}
 Notice that, from the reflection formula,
 we have $S_1(z)=2\sin(\pi z)=\baS_1(z)$.
 We remark that it is shown in \cite[Theorem~2.14]{KK1}
 that $\baS_n(z)$ can be expressed as a product of
 $S_j(z)$ for $j=1,2,\ldots,n$ and vice versa.

\begin{thm}
\label{thm:fegamma}
 We have
 \begin{equation}
\label{for:fegamma}
 \phi_{r}(1-s)=\Bigl(\prod^{2r}_{k=1}S_k(s)^{\a_{r,k}(s-\frac{1}{2})}\Bigr)\phi_r(s)
\end{equation}
 for all 
\ \\[-20pt]
\begin{itemize}
 \item[$\mathrm{(i)}$] $s\in\bC$ when $r=1$. \ \\[-20pt]
 \item[$\mathrm{(ii)}$] $s\in\bC\setminus\bigl((-\infty,-1]\cup [2,+\infty)\cup [0,-i\infty)\cup[1,+i\infty)\bigr)$ 
 when $r\ge 2$. \ \\[-13pt]
\end{itemize}
\end{thm}

 To obtain this, we need the following lemmas.

\begin{lem}
 We have 
\begin{equation}
\label{for:G1-n}
 \G_n(1-z)
=\prod^{n}_{j=1}\bigl(S_j(z)\G_j(z)\bigr)^{(-1)^j\binom{n-1}{j-1}}. 
\end{equation}
\end{lem}
\begin{proof}
 Let $E_n(z):=S_n(z)\G_n(z)$. 
 Note that, from the ladder relations
\begin{align}
\label{for:laddergamma}
 \G_n(z+1)&=\G_n(z)\G_{n-1}(z)^{-1},\\
\label{for:laddersine}
  S_n(z+1)&=S_n(z)S_{n-1}(z)^{-1},
\end{align}
 we have $E_n(z+1)=E_n(z)E_{n-1}(z)^{-1}$.
 Here, we put $\G_0(z):=z^{-1}$ and $S_0(z):=-1$.
 By the definition \eqref{def:multsine}, we have $E_n(z)^{(-1)^n}=\G_n(n-z)$. 
 Hence, using the relation \eqref{for:laddergamma} repeatedly,
 we have  
\begin{align*}
 E_n(z)^{(-1)^n}=\G_n(n-z)
&=\G_{n}(1-z)\prod^{n-2}_{m=0}\G_{n-1}\bigl(n-1-(z+m)\bigr)^{-1}\\
&=\G_{n}(1-z)\prod^{n-2}_{m=0}E_{n-1}(z+m)^{(-1)^{n}}.
\end{align*}
 Therefore, using the equation
 $E_{n-1}(z+m)=\prod^{n-1}_{j=n-1-m}E_j(z)^{(-1)^{n-1-j}\binom{m}{n-1-j}}$,
 which is obtained from the ladder relation of $E_{n}(z)$
 and the formula  $\sum^{b}_{m=a}\binom{m}{a}=\binom{b+1}{b-a}$,
 one sees that  
\begin{align*}
 \G_{n}(1-z)
&=E_n(z)^{(-1)^n}\Biggl(\prod^{n-2}_{m=0}E_{n-1}(z+m)^{(-1)^{n}}\Biggr)^{-1}\\
&=E_n(z)^{(-1)^n}\Biggl(\prod^{n-2}_{m=0}\prod^{n-1}_{j=n-1-m}E_j(z)^{(-1)^{j+1}\binom{m}{n-1-j}}\Biggr)^{-1}\\
&=E_n(z)^{(-1)^n}\prod^{n-1}_{j=1}E_j(z)^{(-1)^{j}\sum^{n-2}_{m=n-1-j}\binom{m}{n-1-j}}\\
&=\prod^{n}_{j=1}E_j(z)^{(-1)^{j}\binom{n-1}{j-1}}.
\end{align*}
 This shows the claim.
\end{proof}

\begin{lem}
 For $1\le k\le 2r$, we have
\begin{equation}
 \label{for:alpha-inv}
 \a_{r,k}(t)=(-1)^k\sum^{2r}_{j=k}\binom{j-1}{k-1}\a_{r,j}(t).
\end{equation}
\end{lem}
\begin{proof}
 Write the righthand-side of \eqref{for:alpha-inv} as $\wt{\a}_{r,k}(t)$.
 It suffices to show that
 $ \sum^{2r}_{k=1}\a_{r,k}(t)\binom{T+j-1}{j-1}
=\sum^{2r}_{k=1}\wt{\a}_{r,k}(t)\binom{T+j-1}{j-1}$.
 In fact, from \eqref{def:alpha}, using the formula \eqref{for:genec},
 we have
\[
 \sum^{2r}_{k=1}\a_{r,k}(t)\binom{T+j-1}{j-1}
=4t^{2r-2}\bigl(T+\frac{1}{2}\bigr)\Bigl(1-\frac{(T+\frac{1}{2})^2}{t^2}\Bigr)^{r-1}.
\]
 On the other hand, using the identity
 $\sum^{j}_{k=1}(-1)^k\binom{j-1}{k-1}\binom{T+k-1}{k-1}=-\binom{-T+j-2}{j-1}$,
 we have 
\begin{align*}
 \sum^{2r}_{k=1}\wt{\a}_{r,k}(t)\binom{T+j-1}{j-1}
&=-\sum^{2r}_{j=1}\a_{r,j}(t)\binom{(-T-1)+j-1}{j-1}\\
&=-4t^{2r-2}\bigl(-T-\frac{1}{2}\bigr)\Bigl(1-\frac{(-T-\frac{1}{2})^2}{t^2}\Bigr)^{r-1}\\
&=4t^{2r-2}\bigl(T+\frac{1}{2}\bigr)\Bigl(1-\frac{(T+\frac{1}{2})^2}{t^2}\Bigr)^{r-1}.
\end{align*}
 Hence the claim follows.
\end{proof}

 We now give a proof of Theorem~\ref{thm:fegamma}. 

\begin{proof}
[Proof of Theorem~\ref{thm:fegamma}] 
 We first notice that both functions $t^{2r}$ and $\a_{r,j}(t)$
 are invariant under the transform $s\mapsto 1-s$
 because they are even polynomials in $t$.
 Hence, replacing $s$ with $1-s$ in \eqref{for:gammafactor1}
 and using the formulas \eqref{for:G1-n} and \eqref{for:alpha-inv},
 we have
\begin{align*}
 \phi_{r}(1-s)
&=e^{-\frac{(2r)!!}{r^2(2r-1)!!}t^{2r}}
\prod^{2r}_{j=1}
\Biggl(\prod^{j}_{k=1}\bigl(S_k(s)\G_k(s)\bigr)^{(-1)^k\binom{j-1}{k-1}}\Biggr)^{\a_{r,j}(t)}\\
&=e^{-\frac{(2r)!!}{r^2(2r-1)!!}t^{2r}}
\prod^{2r}_{k=1}\bigl(S_k(s)\G_k(s)\bigr)^{(-1)^k\sum^{2r}_{j=k}\binom{j-1}{k-1}\a_{r,j}(t)}\\
&=\Bigl(\prod^{2r}_{k=1}S_k(s)^{\a_{r,k}(t)}\Bigr)\phi_r(s).
\end{align*}
 This shows the equation \eqref{for:fegamma}.
\end{proof}

\subsection{Proof of Proposition~\ref{prop:Pa}}
\label{ssec:series}

 Let $m\in\bN_{0}$.
 The aim of this subsection is to give a proof of Proposition~\ref{prop:Pa}. 
 Let
\[
 R_m(t,z):
=\sum^{\infty}_{j=1}\frac{\z(2j+1,z)}{2j+m+1}t^{2j+m+1}
\qquad (|t|<|z|).
\]
 Note that $R_m(t)=R_m(t,\frac{1}{2})$.
 We start from the identity (\cite[p.159 (4)]{SrivastavaChoi2001})
\[
 R_0(t,z)
=\sum^{\infty}_{j=1}\frac{\z(2j+1,z)}{2j+1}t^{2j+1}
=\frac{1}{2}
\Bigl(\log{\G(z-t)}-\log{\G(z+t)}\Bigr)
+t\psi(z)
\qquad (|t|<|z|).
\]
 Letting $z=\frac{1}{2}$, we have  
\begin{align}
\label{for:Rm-middle}
 R_m(t)
&=\int^{t}_{0}\xi^m\frac{d}{d\xi}R_0\bigl(\xi,\frac{1}{2}\bigr)d\xi\\
&=-\frac{1}{2}\int^{t}_{0}\xi^m
\Bigl(\psi\bigl(\frac{1}{2}-\xi\bigr)
+\psi\bigl(\frac{1}{2}+\xi\bigr)\Bigr)d\xi
+\frac{1}{m+1}\psi(\frac{1}{2})t^{m+1}\nonumber\\
&=-\Phi_m\bigl(t,\frac{1}{2}\bigr)
-\frac{1}{2}\log{\baC_{m+1}(t)}
+\frac{1}{m+1}\psi(\frac{1}{2})t^{m+1},\nonumber
\end{align}
 where
\[
 \Phi_{m}(t,z):=\int^{t}_{0}\xi^m\psi(\xi+z)d\xi.
\]
 Notice that, in the last equality, we have used the formula
 $\psi(\frac{1}{2}-\xi)=\psi(\frac{1}{2}+\xi)-\pi\tan(\pi \xi)$
 and \eqref{for:int_multtrig}. 
 Hence, it is enough to evaluate the integral $\Phi_{m}(t,z)$.
 Define the polynomials ${}_nA_{-k}(z)$ for $1\le k\le n-1$ and
 ${}_nB_{m}(z)$ for $m\in\bN_{0}$ by the generating function
\begin{align*}
 \frac{t e^{(n-z)t}}{(e^{t}-1)^n}
=\sum^{n-1}_{k=1}(-1)^k{}_nA_{-k}(z)t^{-k}
+\sum^{\infty}_{m=0}(-1)^m{}_nB_{m}(z)\frac{t^m}{m!}.
\end{align*} 
 These are called the Barnes multiple Bernoulli polynomials (\cite{Barnes}).
 Notice that the degree of ${}_nB_{m}(z)$ is $m+n-1$.
 Since ${}_1B_m(z)=B_m(z)$,
 ${}_nB_{m}(z)$ gives a generalization of the Bernoulli polynomial.
 In fact, using the polynomial ${}_nB_m(z)$,
 one can evaluate the special values of
 the Barnes multiple zeta function $\z_n(w,z)$ at non-positive integer points;
\begin{equation}
\label{for:sv-multiple}
 \z_{n}(1-m,z)
=-\frac{{}_nB_{m}(z)}{m}
 \qquad (m\in\bN).
\end{equation}

 To obtain an explicit expression of $\Phi_{m}(t,z)$,
 we first show the following 

\begin{lem}
 It holds that 
\begin{multline}
\label{for:general-int_mGd}
 (m+1)\int^{t}_{0}\xi^m\log{\G_{n,r}(\xi+z)}d\xi\\
=\sum^{m+1}_{k=1}(-1)^{k-1}\binom{m+1}{k}\binom{r+k-1}{k}^{-1}t^{m+1-k}\log{\G_{n,r+k}(t+z)}
+P_{n,r}(t,z;m),
\end{multline} 
 where $P_{n,r}(t,z;m)$ is the polynomial in $t$ of degree $n+r+m$ defined by  
\begin{align}
\label{for:Ppolynomial}
 P_{n,r}(t,z;m)
:&=\sum^{m+1}_{k=1}
\frac{(-1)^kH(r,r+k-1)}{r+k}\binom{m+1}{k}\binom{r+k-1}{k}^{-1}
t^{m+1-k}{}_nB_{r+k}(t+z)\\
&\ \ \ +(-1)^{m+1}\binom{m+r}{m+1}^{-1}
\Bigl(\log{\G_{n,m+r+1}(z)}-\frac{H(r,m+r)}{m+r+1}{}_nB_{m+r+1}(z)\Bigr)
\nonumber
\end{align} 
 with $H(m,n):=\sum^{n}_{k=m}\frac{1}{k}=H(n)-H(m-1)$ for $n\ge m$.
\end{lem}
\begin{proof}
 Integration by parts yields   
\begin{equation*}
\label{for:iterating}
 \int^{t}_{0}\xi^m(\xi+\a)^{-w}d\xi
=\sum^{m+1}_{k=1}
\frac{(-1)^{k-1}m!}{(m+1-k)!}\frac{t^{m+1-k}(t+\a)^{-w+k}}{(-w+1)_k}
+(-1)^{m+1}m!\frac{\a^{-w+m+1}}{(-w+1)_{m+1}}.
\end{equation*}
 Hence, for $\Re(w)>m+n+1$,
 changing the order of the integral and summation,
 we have 
\begin{align*}
 \int^{t}_{0}\xi^m\z_{n}(w,\xi+z)d\xi
&=\sum_{m_1,\ldots,m_n\ge 0}
\int^{t}_{0}\xi^m\bigl(\xi+(m_1+\cdots+m_n+z)\bigr)^{-w}d\xi\\
&=\sum^{m+1}_{k=1}
\frac{(-1)^{k-1}m!}{(m+1-k)!}\frac{t^{m+1-k}\z_n(w-k,t+z)}{(-w+1)_k}
+(-1)^{m+1}m!\frac{\z_n(w-m-1,z)}{(-w+1)_{m+1}}.
\end{align*}
 This gives a meromorphic continuation of the lefthand-side to the whole plane $\bC$.
 Now the desired formula \eqref{for:general-int_mGd}
 immediately follows from the above equation
 by differentiating both sides at $w=1-r$
 together with the following equation obtained from \eqref{for:sv-multiple};
\[
 \frac{\p}{\p w}\Bigl(\frac{\z_n(w-k,z)}{(-s+1)_k}\Bigr)\Bigl|_{w=1-r}
=\frac{1}{(r)_k}
\Bigl(\log{\G_{n,r+k}(z)}-\frac{1}{r+k}H(r,r+k-1){}_nB_{r+k}(z)\Bigr).
\] 
\end{proof}

\begin{cor}
\label{cor:Phi}
 It holds that 
\begin{equation}
\label{for:explicit-Phi}
 \Phi_m(t,z)
=\sum^{m+1}_{k=1}(-1)^{k+1}\binom{m}{k-1}t^{m+1-k}\log{\mG_{k}(t+z)}
+P_{m}(t,z),
\end{equation}
 where $P_m(t,z)$ is the polynomial in $t$ of degree $m+1$ defined by
\begin{align}
\label{for:P} 
 P_m(t,z):
&=\frac{1}{m+1}H(m)t^{m+1}
+\sum^{m}_{l=1}\frac{(-1)^{l+m}B_{m+1-l}(z)}{l(m+1-l)}t^{l}
+(-1)^{m+1}\log{\mG_{m+1}(z)}.
\end{align}
\end{cor}
\begin{proof}
 From the formula \eqref{for:general-int_mGd}, we have 
\begin{align*}
 \Phi_m(t,z)
&=t^{m}\log{\mG_1(t+z)}-m\int^{t}_{0}\xi^{m-1}\log{\G_{1,1}(\xi+z)}d\xi\\
&=\sum^{m+1}_{k=1}(-1)^{k+1}\binom{m}{k-1}t^{m+1-k}\log{\mG_{k}(t+z)}
-P_{1,1}(t,z;m-1).
\end{align*}
 Here, we have used the identity $\G_{1,1}(z)=\mG_1(z)=\frac{\G(z)}{\sqrt{2\pi}}$.
 Hence it suffices to show that 
\begin{equation}
\label{for:PP}
 -P_{1,1}(t,z;m-1)=P_m(t,z).
\end{equation}
 From the definition \eqref{for:Ppolynomial},
 using the identity $B_m(t+z)=\sum^{m}_{l=0}\binom{n}{l}B_l(z)t^{m-l}$
 and changing the order of the summations, 
 we have
\begin{align*}
 -P_{1,1}(t,z;m-1)
&=\frac{(-1)^{m+1}}{m+1}\sum^{m+1}_{l=0}\binom{m+1}{l}d_l(m)B_{m+1-l}(z)t^{l}\\
&\ \ \ +\frac{(-1)^{m}H(m)}{m+1}B_{m+1}(z)+(-1)^{m+1}\log{\mG_{m+1}(z)},
\end{align*}
 where $d_l(m):=\sum^{l}_{k=0}\binom{l}{k}(-1)^kH(m-k)$. 
 This implies that, to obtain the equation \eqref{for:PP},
 it is enough to show that 
\begin{equation*}
 d_l(m)=
\begin{cases}
 \DS{H(m)} & (l=0),\\[7pt]
 \DS{(-1)^{l-1}\frac{(m-l)!(l-1)!}{m!}} & (1\le l\le m),\\[7pt]
 \DS{(-1)^{m+1}H(m)} & (l=m+1).
\end{cases}
\end{equation*}
 Actually, the case $l=0$ is clear. 
 Let $1\le l\le m+1$ and $F_{l}(s):=\sum^{\infty}_{m=0}d_l(m)s^m$ a generating function of $d_l(m)$.
 Then, from the identity $\sum^{\infty}_{m=1}H(m)s^m=-(1-s)^{-1}\log{(1-s)}$, we have 
\begin{align*}
 F_{l}(s)
=\sum^{l}_{k=0}\binom{l}{k}(-s)^k\sum^{\infty}_{m=1}H(m)s^m
=-(1-s)^{l-1}\log{(1-s)}.
\end{align*}
 Hence, by the Leibniz rule, we have  
\begin{align}
\label{for:dtilde}
 d_l(m)=\frac{1}{m!}\frac{d^m}{ds^m}F_{l}(s)\Bigl|_{s=0}
&=-\frac{1}{m!}\sum^{m}_{k=0}\binom{m}{k}
\frac{d^{m-k}}{ds^{m-k}}\bigl((1-s)^{l-1}\bigr)
\frac{d^{k}}{ds^k}\bigl(\log{(1-s)}\bigr)\Bigl|_{s=0}\\
&=\sum^{m}_{k=\max\{m+1-l,1\}}\frac{(-1)^{m-k}}{k}\binom{l-1}{m-k}.\nonumber
\end{align}
 When $1\le l\le m$, this is equal to 
\[
 \sum^{l-1}_{k=0}\frac{(-1)^k}{m-k}\binom{l-1}{k}
=\int^{1}_{0}x^{m-1}(1-x^{-1})^{l-1}dx
=(-1)^{l-1}\frac{(m-l)!(l-1)!}{m!}.
\]
 On the other hand, when $l=m+1$, using the formula $\psi(m)=-\g+H(m-1)$ for $m\in\bN$
 and the equation (see, e.g., \cite[p.15 (13)]{SrivastavaChoi2001})
\begin{equation*}
 \int^{1}_{0}\frac{(1-\xi)^{z-1}-1}{\xi}d\xi=-\g-\psi(z)
 \qquad (\Re(z)>0),
\end{equation*}
 where $\g=0.57721\ldots$ is the Euler constant,  
 we see that the rightmost-hand side of \eqref{for:dtilde} equals
\[
 (-1)^m\sum^{m}_{k=1}\frac{(-1)^k}{k}\binom{m}{k}
=(-1)^m\int^{1}_{0}\frac{(1-x)^m-1}{x}dx
=(-1)^{m+1}H(m).
\]
 This completes the proof.
\end{proof}

 We now give the proof of Proposition~\ref{prop:Pa}. 

\begin{proof}
[Proof of  Proposition~\ref{prop:Pa}]
 Notice that, since $m$ is odd, $B_{m+1-l}(\frac{1}{2})=0$ if $l$ is odd.
 Hence one obtains the formula \eqref{for:Rm} from \eqref{for:Rm-middle}
 together with the formulas \eqref{for:explicit-Phi} and \eqref{for:P} with $z=\frac{1}{2}$.
\end{proof}

\section{Milnor-Selberg zeta functions $Z_{\G,r}(s)$} 
\label{sec:zeta factor}

 In this final section,
 we study the Milnor-Selberg zeta function $Z_{\G,r}(s)$.
 Recall that $Z_{\G,r}(s)$ is defined by \eqref{def:zeta}
 and is holomorphic in the region $U$ (see Figure~$4$). 

\subsection{Analytic properties of $Z_{\G,r}(s)$}

 The following theorem is easily obtained from the earlier discussion.
\begin{thm}
\label{thm:analyticpropZ}
 The function $Z_{\G,r}(s)$
 can be continued analytically to
\ \\[-20pt]
\begin{itemize}
 \item[$\mathrm{(i)}$] the whole plane $\bC$ when $r=1$.\ \\[-20pt]
 \item[$\mathrm{(ii)}$] $\Omega_{\G}\setminus (-\infty,-1]$ when $r\ge 2$.\ \\[-20pt]
\end{itemize}
 Moreover, it satisfies the functional equation 
\begin{equation}
\label{for:feMS}
 Z_{\G,r}(1-s)
=\Bigl(\prod^{2r}_{j=1}S_j(s)^{-\a_{r,j}(s-\frac{1}{2})}\Bigr)^{g-1}Z_{\G,r}(s)
\end{equation}
 for all
\ \\[-20pt]
\begin{itemize}
 \item[$\mathrm{(i)}$] $s\in\bC$ when $r=1$. \ \\[-20pt]
 \item[$\mathrm{(ii)}$] $s\in\Omega_{\G}\setminus\bigl((-\infty,-1]\cup[2,+\infty)\bigr)$
 when $r\ge 2$.\ \\[-20pt]
\end{itemize}
\end{thm}
\begin{proof}
 From \eqref{for:det_r2},
 we have 
\begin{equation}
\label{for:ZGD}
 Z_{\G,r}(s)=\phi_r(s)^{-(g-1)}\mathrm{D}_{\G,r}(s).
\end{equation}
 This gives the desired analytic continuation
 from Theorem~\ref{thm:HDDs} and Theorem~\ref{thm:gamma factor1}. 
 The functional equation \eqref{for:feMS} immediately follows from
 Theorem~\ref{thm:HDDs} and Theorem~\ref{thm:fegamma}.
\end{proof}

 To study a ``complete Milnor-Selberg zeta function'', we need the following lemma.

\begin{lem}
 For $0\le l\le 2r-1$, let 
\[
 \hat{\a}_{r,l}(t):=(-1)^{l}\sum^{2r-l}_{j=1}\binom{2r-j}{l}\a_{r,j}(t).
\]
 Then, for $1\le m\le 2r$, we have 
\begin{equation}
\label{for:checka2}
 \sum^{2r-1}_{l=2r-m}\binom{l}{2r-m}\hat{\a}_{r,l}(t)=(-1)^m\a_{r,m}(t).
\end{equation}
\end{lem}
\begin{proof}
 Using the equation $\sum^{b}_{l=a}(-1)^l\binom{l}{a}\binom{b}{l}=\d_{a,b}(-1)^a$
 where $\d_{a,b}$ is the Kronecker delta, we have 
\begin{align*}
 \sum^{2r-1}_{l=2r-m}\binom{l}{2r-m}\hat{\a}_{r,l}(t)
=\sum^{m}_{j=1}\Biggl(\sum^{2r-j}_{l=2r-m}(-1)^l\binom{l}{2r-m}\binom{2r-j}{l}\Biggr)\a_{r,j}(t)
=(-1)^m\a_{r,m}(t).
\end{align*}
\end{proof}

 The following gives a generalization of the functional equation \eqref{for:fe2}.

\begin{cor}
\label{cor:completeMS}
 Define the complete Milnor-Selberg zeta function by
\begin{equation}
\label{def:completeMS}
 \Xi_{\G,r}(s)
:=\Bigl(\prod^{2r-1}_{l=0}\G_{2r}(s+l)^{\hat{\a}_{r,l}(s-\frac{1}{2})}\Bigr)^{g-1}Z_{\G,r}(s).
 \end{equation} 
 Then, we have
\begin{equation}
\label{for:completeMS}
 \Xi_{\G,r}(s)
=e^{\frac{(2r)!!}{r^2(2r-1)!!}(g-1)(s-\frac{1}{2})^{2r}}\cdot \mathrm{D}_{\G,r}(s).
\end{equation}
 In particular, $\Xi_{\G,r}(s)$ is
\ \\[-20pt]
\begin{itemize}
 \item[$(\mathrm{i})$] an entire function when $r=1$.\ \\[-20pt]
 \item[$(\mathrm{ii})$] a holomorphic function in $\Omega_{\G}$ when $r\ge 2$.\ \\[-20pt]
\end{itemize}
 Moreover, it satisfies the functional equation 
\begin{equation*}
 \Xi_{\G,r}(1-s)=\Xi_{\G,r}(s)
\end{equation*}
 for all
\ \\[-20pt]
\begin{itemize}
 \item[$\mathrm{(i)}$] $s\in\bC$ when $r=1$. \ \\[-20pt]
 \item[$\mathrm{(ii)}$] $s\in\Omega_{\G}$ when $r\ge 2$.\ \\[-20pt]
\end{itemize}
\end{cor}
\begin{proof}
 It is sufficient to prove the equation \eqref{for:completeMS}.
 To do that, from \eqref{for:det_r2} with \eqref{for:gammafactor1}, 
 it is enough to show that
 $\prod^{2r-1}_{l=0}\G_{2r}(s+l)^{\hat{\a}_{r,l}(t)}=\prod^{2r}_{m=1}\G_{m}(s)^{\a_{r,m}(t)}$.
 Actually, from the ladder relation \eqref{for:laddergamma},
 one can show that $\G_{2r}(s+l)=\prod^{l}_{j=0}\G_{2r-j}(s)^{(-1)^j\binom{l}{j}}$. 
 This yields
\begin{align*}
 \prod^{2r-1}_{l=0}\G_{2r}(s+l)^{\hat{\a}_{r,l}(t)}
&=\prod^{2r-1}_{l=0}\prod^{l}_{j=0}\G_{2r-j}(s)^{(-1)^j\binom{l}{j}\hat{\a}_{r,l}(t)}\\
&=\prod^{2r}_{m=1}\G_m(s)^{(-1)^m\sum^{2r-1}_{l=2r-m}\binom{l}{2r-m}\hat{\a}_{r,l}(t)}\\
&=\prod^{2r}_{m=1}\G_m(s)^{\a_{r,m}(t)}.
\end{align*}
 In the last equality, we have used the equation \eqref{for:checka2}.
 This completes the proof.
\end{proof}

\begin{example}
 Let $t=s-\frac{1}{2}$.
 Then, we have 
\begin{align*}
 \Xi_{\G,1}(s)&=\bigl(\G_2(s)^{2}\G_2(s+1)^2\bigr)^{g-1}Z_{\G,1}(s)=\Xi_{\G}(s),\\
 \Xi_{\G,2}(s)&=\bigl(\G_4(s)^{-\frac{1}{2}+2t^2}\G_4(s+1)^{-\frac{23}{2}-2t^2}\G_4(s+2)^{-\frac{23}{2}-2t^2}\G_4(s+3)^{-\frac{1}{2}+2t^2}\bigr)^{g-1}Z_{\G,2}(s),\\
 \Xi_{\G,3}(s)&=\bigl(\G_6(s)^{\frac{1}{8}-t^2+2t^4}\G_6(s+1)^{\frac{237}{8}-21t^2-6t^4}\G_6(s+2)^{\frac{841}{4}+22t^2+4t^4}\\
&\ \ \ \times\G_6(s+3)^{\frac{841}{4}+22t^2+4t^4}\G_6(s+4)^{\frac{237}{8}-21t^2-6t^4}\G_6(s+5)^{\frac{1}{8}-t^2+2t^4}\bigr)^{g-1}Z_{\G,3}(s).
\end{align*}
\end{example}

\begin{remark}
 Using the ladder relations \eqref{for:laddergamma} and \eqref{for:laddersine},
 one can prove that 
\[
 \prod^{2r}_{j=1}S_j(s)^{\a_{r,j}(t)}
=\frac{\prod^{2r-1}_{l=0}\G_{2r}(1-s+l)^{\hat{\a}_{r,j}(t)}}
{\prod^{2r-1}_{l=0}\G_{2r}(s+l)^{\hat{\a}_{r,j}(t)}}.
\]  
 This reads the definition \eqref{def:completeMS} of the complete Milnor-Selberg zeta function.
\end{remark}

\begin{remark}
 The equation~\eqref{for:completeMS} implies that 
 we can get rid of the singularities of $Z_{\G,r}(s)$ in $(-\infty,-1]$
 by multiplying suitable gamma factors.  
 Moreover, it also says that
 the Milnor-Selberg zeta function $Z_{\G,r}(s)$ has no ``non-trivial zeros''
 because $\mathrm{D}_{\G,r}(s)$ does have no zeros.
\end{remark}

\begin{remark}
\label{rem:multpoly1}
 Let $f(s)$ be a function on $\bC$ and $m(s)$ a polynomial.
 Suppose that it can be written as $f(s)=(s-a)^{m(s)}g(s)$ around $s=a$
 where $g(s)$ is a holomorphic function at $s=a$ with $g(a)\ne 0$.
 In this case, let us say that $f(s)$ has a ``multiplicity polynomial $m(s)$ at $s=a$''
 and write $m(s)=m(s;f,a)$.
 It is clear that this is a generalization of the multiplicity (or order)
 of zeros or poles of meromorphic functions (they are the case $m(s)\in\bZ$).
 For example, if $f(s)$ has a zero (resp. a pole) of order $n$ at $s=a$,
 then the multiplicity polynomial of $f(s)^{q(s)}$ for a polynomial $q(s)$ at $s=a$ is given by
 $m(s;f^q,a)=nq(s)$ (resp. $-nq(s)$).

 From the expression \eqref{for:ZGD}, for $k\in\bN$,
 one can see that 
\begin{equation}
\label{for:mpforZr}
 m(s;Z_{\G,r},-k)=2(g-1)(2k+1)(s-k)^{r-1}(s-k-1)^{r-1}.
\end{equation}
 In particular, when $r=1$,
 this coincides with the multiplicity $2(g-1)(2k+1)$ 
 of the trivial zero $s=-k$ of the Selberg zeta function $Z_{\G}(s)$. 
 The equation \eqref{for:mpforZr} is obtained as follows;
 from the equation \eqref{for:ZGD} and the expression \eqref{for:gammafactor1}
 together with the fact that $\G_j(s)^{-1}$ has a zero at $s=-k$ of order $\binom{k+j-1}{j-1}$, 
 it can be written as 
\[
  Z_{\G,r}(s)
=\Biggl(\prod^{2r}_{j=1}(s+k)^{(g-1)\binom{k+j-1}{j-1}\a_{r,j}(t)}\Biggr)\cdot z_{\G,r}(s),
\]
 where $z_{\G,r}(s)$ is some holomorphic function at $s=-k$ with $z_{\G,r}(-k)\ne 0$.
 This shows that 
\begin{align*}
 m(s;Z_{\G,r},-k)
&=(g-1)\sum^{2r}_{j=1}\binom{k+j-1}{j-1}\a_{r,j}(t)\\
&=4(g-1)\sum^{r}_{l=1}\Biggl[\sum^{2\ell}_{j=1}c_{2\ell,j}\bigl(\frac{1}{2}\bigr)\binom{k+j-1}{j-1}\Biggr]
\binom{r-1}{\ell-1}(-1)^{\ell-1}t^{2r-2\ell}\\
&=2(g-1)(2k-1)t^{2r-2}\sum^{r-1}_{\ell=0}\binom{r-1}{\ell}\Bigl(-\bigl(k+\frac{1}{2}\bigr)^2t^{-2}\Bigr)^{\ell}\\
&=2(g-1)(2k-1)\Bigl(t^2-\bigl(k+\frac{1}{2}\bigr)^2\Bigr)^{r-1}.
\end{align*}
 Note that, in the third equality,
 we have employed the formula \eqref{for:genec}.  
\end{remark}


\subsection{Poly-Selberg zeta functions}

 In order to study an Euler product expression of the Milnor-Selberg zeta function $Z_{\G,r}(s)$,
 we now introduce a certain generalization of the Selberg zeta function.

 For $m\in\bN$,
 define the function $Z_{\G}^{(m)}(s)$ by the following Euler product.
\begin{equation}
\label{def:poly-Selberg}
 Z_{\G}^{(m)}(s)
:=\prod_{P\in\Prim(\G)}\prod^{\infty}_{n=0}
 H_m\bigl(N(P)^{-s-n}\bigr)^{(\log N(P))^{-m+1}},
\end{equation}
 where $H_{m}(z):=\exp(-Li_m(z))$ with 
 $Li_m(z):=\sum^{\infty}_{k=1}\frac{z^k}{k^m}$ being the polylogarithm of degree $m$.
 This infinite product converges absolutely for $\Re(s)>1$.
 We call $Z_{\G}^{(m)}(s)$ a {\it poly-Selberg zeta function} of degree $m$.
 Notice that, since $Li_1(z)=-\log{(1-z)}$ and hence $H_{1}(z)=1-z$,
 we have $Z_{\G}^{(1)}(s)=Z_{\G}(s)$.
 To give an analytic continuation of $Z_{\G}^{(m)}(s)$, 
 we first show the following

\begin{lem}
 It holds that 
\begin{align}
\label{for:logderiZm}
 -\log{Z^{(m)}_{\G}(s)}
=\sum_{\g\in\Hyp(\G)}
\frac{\log N(\d_{\g})}{N(\g)^{\frac{1}{2}}-N(\g)^{-\frac{1}{2}}}
\frac{N(\g)^{-s+\frac{1}{2}}}{(\log N(\g))^m}
\qquad (\Re(s)>1).
\end{align}
\end{lem}
\begin{proof}
 Note that $N(P^k)=N(P)^k$ and hence $\log N(P^k)=k\log N(P)$
 for $P\in\Prim(\G)$ and $k\in\bN$.
 Hence, from the definition \eqref{def:poly-Selberg}, we have 
\begin{align*}
-\log{Z^{(m)}_{\G}(s)}
&=\sum_{P\in\Prim(\G)}\sum^{\infty}_{n=0}
\sum^{\infty}_{k=1}\frac{N(P)^{-(s+n)k}}{k^{m}(\log N(P))^{m-1}}\\
&=\sum_{P\in\Prim(\G)}\sum^{\infty}_{k=1}\frac{N(P)^{-ks}}
{k^m(\log{N(P)})^{m-1}}\frac{1}{1-N(P)^{-k}}\\
&=\sum_{P\in\Prim(\G)}\sum^{\infty}_{k=1}
 \frac{\log{N(P)}}{N(P^k)^{\frac{1}{2}}-N(P^k)^{-\frac{1}{2}}}\frac{N(P^{k})^{-s+\frac{1}{2}}}{(\log N(P^k))^{m}}.
\end{align*}
 Writing $P^{k}=\g$,
 one obtains the desired expression.
\end{proof}

 The poly-Selberg zeta function satisfies a differential ladder relation.
 As a consequence, when $m\ge 2$,
 one can obtain an iterated integral representation of $Z^{(m)}_{\G}(s)$ 
 which gives an analytic continuation beyond the line $\Re(s)=1$.

\begin{prop}
\label{IteratedSelberg}
 $(\mathrm{i})$\ It holds that  
\begin{align}
\label{for:partialzeta}
 \frac{d^{m-1}}{ds^{m-1}}\log Z_{\G}^{(m)}(s)
=(-1)^{m-1}\log Z_{\G}(s)
\qquad (\Re(s)>1).
\end{align}
 $(\mathrm{ii})$\ Fix $a\in\bC$ with $\Re(a)>1$. Then, for $m\ge 2$, we have  
\begin{align}
\label{for:IntRep}
 Z_{\G}^{(m)}(s)
=Q^{(m)}_{\G}(s,a)\exp\Biggl(
 \underbrace{\int^{s}_{a}\int^{\xi_{m-1}}_{a}\cdots\int^{\xi_2}_{a}}_{m-1}
 \log Z_{\G}(\xi_1)d\xi_1\cdots d\xi_{m-1}\Biggl)^{(-1)^{m-1}},
\end{align}
 where $Q^{(m)}_{\G}(s,a):=\prod^{m-2}_{k=0}Z_{\G}^{(m-k)}(a)^{\frac{(-1)^k}{k!}(s-a)^k}$.
 If $\frac{1}{4}\notin\Spec(\Delta_{\G})$,
 then this gives an analytic continuation of $Z_{\G}^{(m)}(s)$
 to the region $\Omega_{\G}\setminus (-\infty,-1]$,
 otherwise, to the region
 $\Omega^{+}_{\G}:=\Omega_{\G}\cap\{s\in\bC\,|\,\Re(s)>\frac{1}{2}\}$.
\end{prop}
\begin{proof}
 We first show the equation \eqref{for:partialzeta}.
 The case $m=1$ is clear because $Z_{\G}^{(1)}(s)=Z_{\G}(s)$.
 Let $m\ge 2$. From \eqref{for:logderiZm}
 (or the differential equation $\frac{d}{dz}Li_m(z)=\frac{1}{z}Li_{m-1}(z)$),
 we have  
\begin{equation}
\label{for:deri}
 \frac{d}{ds}\log{Z^{(m)}_{\G}(s)}=-\log{Z^{(m-1)}_{\G}(s)}.
\end{equation}  
 Using this equation repeatedly, 
 one obtains the equation \eqref{for:partialzeta}.
 We next show $(\mathrm{ii})$ by induction on $m$.
 Assume that $\frac{1}{4}\notin\Spec(\Delta_{\G})$.
 Let $m=2$.
 Then, integrating the equation \eqref{for:partialzeta} with $m=2$,
 we have 
\begin{align}
\label{for:m=2}
 \log Z_{\G}^{(2)}(s)
=\log{Z^{(2)}_{\G}(a)}-\int^{s}_{a}\log Z_{\G}(\xi)d\xi.
\end{align}
 Here, we take the path in $\Re(s)>1$.
 This immediately shows the equation \eqref{for:IntRep} with $m=2$ for $\Re(s)>1$.
 Here, in the righthand-side of \eqref{for:m=2},
 one can move $s$ freely in the region
 in where $\log{Z_{\G}(s)}$ is single-valued and holomorphic, that is,
 $\Omega_{\G}\setminus(-\infty,-1]$
 (notice that the Selberg zeta function $Z_{\G}(s)$ has zeros at $s=1,0,-k$ for $k\in\bN$ and
 $s=\a^{\pm}_j$ for $j\in\bN$).
 Hence the equation \eqref{for:IntRep} gives
 an analytic continuation of $Z_{\G}^{(2)}(s)$ to $\Omega_{\G}\setminus(-\infty,-1]$.
 Now, assume that the claim $(\mathrm{ii})$ holds for $m-1$. 
 Then, we have 
\begin{align}
\label{for:m-1}
 \log Z_{\G}^{(m-1)}(\xi)
&=\sum^{m-1}_{k=0}\frac{(-1)^k}{k!}(\xi-a)^k\log{Z^{(m-1-k)}_{\G}(a)}\\
&\ \ \ +(-1)^{m-2}\underbrace{\int^{\xi}_{a}\int^{\xi_{m-2}}_{a}
\cdots\int^{\xi_2}_{a}}_{m-2} \log Z_{\G}(\xi_1)d\xi_1\cdots d\xi_{m-2}.\nonumber
\end{align}
 Taking the integral of the equation \eqref{for:m-1}
 together with the formula \eqref{for:deri},
 we have 
\begin{align*}
\log Z_{\G}^{(m)}(s)
&=\log{Z^{(m)}_{\G}(a)}+\sum^{m-3}_{k=0}\frac{(-1)^{k+1}}{(k+1)!}(s-a)^{k+1}\log{Z^{(m-(k+1))}_{\G}(a)}\\
&\ \ \ +(-1)^{m-2}\int^{s}_{a}\underbrace{\int^{\xi}_{a}\int^{\xi_{m-2}}_{a}
\cdots\int^{\xi_1}_{a}}_{m-2} \log Z_{\G}(\xi_1)d\xi_1\cdots d\xi_{m-2}d\xi.
\end{align*}
 Therefore, by the same discussion as above,
 one shows the claim $(\mathrm{ii})$ for all $m\ge 2$.
 The proof works similarly for the case $\frac{1}{4}\in\Spec(\Delta_{\G})$
 (notice that $\Omega^{+}_{\G}$ is the connected component of $\Omega_{\G}$
 which includes the region $\{s\in\bC\,|\,\Re(s)>1\}$).
\end{proof}


\begin{remark}
 The discussion above can be applied to a general scheme for arbitrary number fields.
 In fact, let $L(s,\pi)$ be a $L$-functions attached to
 irreducible cuspidal automorphic representation $\pi$ of $\GL_d(\bA_{K})$,
 where $K$ is an algebraic number field and $\bA_K$ is the adele ring of $K$. 
 We have shown in \cite{WakayamaYamasaki} that,
 as a generalization of the result in \cite{Deninger1992},
 a ``higher depth regularized products'' of the
 non-trivial zeros of $L(s,\pi)$   
 can be evaluated as a product of the Milnor gamma functions and 
 ``poly $L$-functions'',
 which are similarly defined by an Euler product
 as the poly-Selberg zeta function. 
\end{remark}

\subsection{Euler product expressions of $Z_{\G,r}(s)$}

 The following formula can be regarded as an Euler product expression of $Z_{\G,r}(s)$
 (remark that the poly-Selberg zeta function $Z^{(m)}_{\G}(s)$
 is defined by the Euler product \eqref{def:poly-Selberg}).

\begin{thm}
\label{thm:zeta factor}
 It holds that   
\begin{equation}
\label{for:zetafunction}
 Z_{\G,r}(s)
=\Biggl(
\prod^{r-1}_{m=0}Z^{(r+m)}_{\G}(s)^{\frac{(r-1+m)!}{m!(r-1-m)!}(2s-1)^{r-1-m}}
\Biggr)^{(r-1)!(-1)^{r-1}}.
\end{equation}
\end{thm}
\begin{proof}
 We start from the definition \eqref{def:zeta}. 
 Let $s\in U^{+}:=U\cap\{s\in\bC\,|\,\Re(s)>\frac{1}{2}\}$.
 Then, using the formula
\[
 \int^{\infty}_{0}{\xi}^we^{-(a\xi+\frac{b}{\xi})}\frac{d\xi}{\xi}
=2\Bigl(\frac{b}{a}\Bigr)^{\frac{w}{2}}K_{w}(2a^{\frac{1}{2}}b^{\frac{1}{2}})
\qquad (\Re(a)>0,\ \Re(b)>0),
\]
 where $K_{\n}(x)$ is the modified Bessel function of the second kind, we have 
\begin{align*}
 \Theta_{\G,r}(w,t)
&=\frac{1}{\sqrt{\pi}\G(w+1-r)}\\
&\ \ \ \times
\sum_{\g\in\Hyp(\G)}
\frac{\log N(\d_\g)}{N(\g)^{\frac{1}{2}}-N(\g)^{-\frac{1}{2}}}
\Bigl(\frac{\log{N({\g})}}{2t}\Bigr)^{w+\frac{1}{2}-r}K_{w+\frac{1}{2}-r}\bigl(t\log{N(\g)}\bigr).
\end{align*} 
 Here, employing the asymptotic formulas
\begin{align*}
\frac{1}{\G(w+1-r)}
&=(-1)^{r-1}(r-1)!w+O(w^2),\\
 \Bigl(\frac{\log{N(\g)}}{2t}\Bigr)^{w+\frac{1}{2}-r}
&=\Bigl(\frac{\log{N(\g)}}{2t}\Bigr)^{\frac{1}{2}-r}+O(w)
\end{align*}
 and
\[
 K_{w+\frac{1}{2}-r}\bigl(t\log{N(\g)}\bigr)
=K_{\frac{1}{2}-r}\bigl(t\log{N(\g)}\bigr)+O(w)
\]
 as $w\to 0$ together with the equation (see, e.g., \cite{EMOT})
\begin{align*}
 K_{\frac{1}{2}-r}(y)=K_{r-\frac{1}{2}}(y)
&=\Bigl(\frac{\pi}{2y}\Bigr)^{\frac{1}{2}}e^{-y}
 \sum^{r-1}_{m=0}(2y)^{-m}\frac{(r-1+m)!}{m!(r-1-m)!},
\end{align*}
 one can see from \eqref{for:logderiZm} that
\begin{align}
\label{for:deriI}
 \log{Z_{\G,r}(s)}
&=-\frac{\p}{\p w}\Theta_{\G,r}(w,t)\Bigl|_{w=0}\\
&=(-1)^{r-1}(r-1)!
\sum_{m=0}^{r-1}\frac{(r-1+m)!}{m!(r-1-m)!}(2t)^{r-1-m}
\log{Z^{(r+m)}_{\G}(s)}.\nonumber
\end{align}
 This shows the expression \eqref{for:zetafunction}.
\end{proof}

\begin{example}
 Let $t=s-\frac{1}{2}$.
 Then, we have 
\begin{align*}
 Z_{\G,1}(s)
&=Z^{(1)}_{\G}(s)=Z_{\G}(s),\\
 Z_{\G,2}(s)
&=Z^{(2)}_{\G}(s)^{-(2t)}Z^{(3)}_{\G}(s)^{-2},\\
 Z_{\G,3}(s)
&=Z^{(3)}_{\G}(s)^{2(2t)^2}Z^{(4)}_{\G}(s)^{12(2t)}Z^{(5)}_{\G}(s)^{24}.
\end{align*}
\end{example}

\begin{remark}
 We remark that, when $r\ge 2$ and $\frac{1}{4}\notin\Spec(\Delta_{\G})$,
 from Proposition~\ref{IteratedSelberg},
 the expression \eqref{for:zetafunction} also gives an analytic continuation
 of $Z_{\G,r}(s)$ to the region
 $\Omega_{\G}\setminus (-\infty,-1]$.
 (We have already known this fact from Theorem~\ref{thm:analyticpropZ}.
 See Figure~$3$ in Section~\ref{sec:introduction}.)
\end{remark}

 Finally, as is the case of $Z^{(m)}_{\G}(s)$,
 we show that the Milnor-Selberg zeta function $Z_{\G,r}(s)$
 also satisfies a differential ladder relation and hence has an iterated integral representation,
 which again gives an analytic continuation to
 $\Omega_{\G}\setminus (-\infty,-1]$ if $\frac{1}{4}\notin\Spec(\Delta_{\G})$. 

\begin{cor}
\label{cor:iteMilnor}
 $(\mathrm{i})$\ It holds that 
\begin{equation}
\label{for:logderiZr}
 \Bigl(\frac{1}{2s-1}\frac{d}{ds}\Bigr)^{r-1}\log{Z_{\G,r}(s)}
=(r-1)!\log{Z_{\G}(s)}
\qquad (s\in U^{+}).
\end{equation} 
 $(\mathrm{ii})$\ Fix $a\in U^{+}$. Then, for $r\ge 2$, we have  
\begin{align}
\label{for:IntRepZr}
 Z_{\G,r}(s)
=Q_{\G,r}(s,a)\exp\Biggl(
 \underbrace{\int^{s}_{a}\int^{\xi_{r-1}}_{a}\cdots\int^{\xi_2}_{a}}_{r-1}
 \Bigl(\prod^{r-1}_{j=1}(2\xi_j-1)\Bigr)\log Z_{\G}(\xi_1)d\xi_1 \cdots d\xi_{r-1}\Biggl)^{(r-1)!},
\end{align}
 where $Q_{\G,r}(s,a):=\prod^{r-2}_{k=0}Z_{\G,r-k}(a)^{\binom{r-1}{k}(s-a)^k(s+a-1)^k}$.
 If $\frac{1}{4}\notin\Spec(\Delta_{\G})$,
 then this gives an analytic continuation of $Z_{\G,r}(s)$ 
 to the region $\Omega_{\G}\setminus (-\infty,-1]$,
 otherwise, to the region $\Omega^{+}_{\G}$.
\end{cor}
\begin{proof}
 To prove the equation \eqref{for:logderiZr},
 it is sufficient to show for $r\ge 2$ that 
\begin{equation}
\label{for:ladderZr}
 \Bigl(\frac{1}{2s-1}\frac{d}{ds}\Bigr)\log{Z_{\G,r}(s)}
=(r-1)\log{Z_{\G,r-1}(s)}.
\end{equation}
 Actually, from the expression \eqref{for:deriI} together with \eqref{for:deri}, 
 one sees that 
\begin{align*}
\frac{d}{ds}\log{Z_{\G,r}(s)}
&=(-1)^{r}(r-1)!\sum^{r-2}_{m=0}\frac{(r+m-2)!}{m!(r-m-2)!}(2s-1)^{r-m-1}
\log{Z^{(m+r-1)}_{\G}(s)}\\
&=(2s-1)(r-1)\log{Z_{\G,r-1}(s)}.
\end{align*} 
 This shows the equation \eqref{for:ladderZr}.
 The iterated integral representation \eqref{for:IntRepZr}
 and the analytic continuation can be easily derived 
 by the similar discussion performed in Proposition~\ref{IteratedSelberg}.
\end{proof}

\begin{Acknowledgement}
 The authors would like to thank the referee and Hirotaka Akatsuka for their careful
 reading of the manuscript and helpful advice.
\end{Acknowledgement}



\bigskip

\noindent
\textsc{Nobushige KUROKAWA}\\
 Department of Mathematics, Tokyo Institute of Technology,\\
 Oh-okayama Meguro, Tokyo, 152-0033 JAPAN.\\
\texttt{kurokawa@math.titech.ac.jp}\\

\noindent
\textsc{Masato WAKAYAMA}\\
 Faculty of Mathematics, Kyushu University,\\
 Motooka, Nishiku, Fukuoka, 819-0395, JAPAN.\\
\texttt{wakayama@math.kyushu-u.ac.jp}\\

\noindent
\textsc{Yoshinori YAMASAKI}\\
 Graduate School of Science and Engineering, Ehime University,\\
 Bunkyo-cho, Matsuyama, 790-8577 JAPAN.\\
 \texttt{yamasaki@math.sci.ehime-u.ac.jp}


\begin{thebibliography}{9999999}
\bibitem[A]{A}
 V.S. Adamchik.:
 The multiple gamma function and its application to computation of
 series,
 {\it The Ramanujan J.}, {\bf 9} (2005), 271--288.
\bibitem[B1]{Barnes1899}
 E.W. Barnes.: 
 The theory of the $G$-function,  
 {\it Quart. J. Math.},
 {\bf 31} (1899), 264--314. 
\bibitem[B2]{Barnes}
 E.W. Barnes.: 
 On the theory of the multiple gamma functions,
 {\it Trans. Cambridge Philos. Soc.},
 {\bf 19} (1904), 374--425.
\bibitem[{D}]{Deninger1992}
 C. Deninger.:
 Local $L$-factors of motives and regularized determinants,
 {\it Invent. Math.}, {\bf 107} (1992), 135--150.
\bibitem[{D}{H}{P}]{DHP}
 E. D'Hoker and D.H. Phong.:
 On determinants of Laplacians on Riemann surfaces,
 {\it Commun. Math. Phys.}, {\bf 104} (1986), 537--545.
\bibitem[{E}{M}{O}{T}]{EMOT}
 A. Erd\'elyi, W. Magnus, F. Oberthettinger and F.G. Tricomi.:
 Higher Transcendental Functions, 
 McGraw-Hill, New York, 1953. 
\bibitem[{K}{S}]{KS1}
 J.P. Keating and N.C. Snaith.:
 Random matrix theory and $\z(\frac{1}{2}+it)$,
 {\it Commun. Math. Phys.}, {\bf 214} (2000), 57--89.
\bibitem[{Kub}]{Kubert}
 D. Kubert.:
 The universal ordinary distribution,
 {\it Bull. Soc. Math. France}, {\bf 107} (1979), 179--202.
\bibitem[{K}{K}]{KK1}
 N. Kurokawa and S. Koyama.:
 Multiple sine functions,
 {\it Forum Math.}, {\bf 15} (2003), 839--876.
\bibitem[{K}{O}{W}1]{KOW1}
 N. Kurokawa, H. Ochiai and M. Wakayama.:
 Multiple trigonometry and zeta functions, 
 {\it J. Ramanujan Math. Soc.}, {\bf 17} (2002), 101--113.
\bibitem[{K}{O}{W}2]{KOW2}
 N. Kurokawa, H. Ochiai and M. Wakayama.:
 Milnor's multiple gamma functions,
 {\it J. Ramanujan Math. Soc.}, {\bf 21} (2006), 153--167.
\bibitem[M]{M}
 J. Milnor.:
 On polylogarithms, Hurwitz zeta functions, and the Kubert identities,
 {\it Enseignement Math\'ematique}, {\bf 29} (1983), 281--322.
\bibitem[R]{Rosenberg1997}
 S. Rosenberg.:
 The Laplacian on a Riemannian manifold.
 An introduction to analysis on manifolds.
 London Mathematical Society Student Texts, {\bf 31}.
 Cambridge University Press, Cambridge, 1997.
\bibitem[S]{Sa}
 P. Sarnak.:
 Determinants of Laplacians,
 {\it Commun. Math. Phys.}, {\bf 110} (1987), 113--120.
\bibitem[{S}{C}]{SrivastavaChoi2001}
 H.M. Srivastava and J. Choi.:
 Series associated with the zeta and related functions,
 Kluwer Academic Publishers, Dordrecht, 2001.
\bibitem[Vo]{V}
 A. Voros.:
 Spectral functions, special functions and the Selberg zeta functions,
 {\it Commun. Math. Phys.}, {\bf 110} (1987), 439--465.
\bibitem[Vi]{Vigneras1979}
 M.F. Vign\'eras.:
 L'\'equation fonctionalie de la fonction zeta de Selberg de groupe
 modulaire $PSL(2,\bZ)$, 
 {\it Ast\'erisque}, {\bf 61} (1979), 235--249.
\bibitem[WY]{WakayamaYamasaki}
 M. Wakayama and Y. Yamasaki.:
 Higher regularizations for zeros of cuspidal automorphic $L$-functions of $\GL_d$, 
 preprint, 2011. arXiv:0909.4925
\bibitem[Y]{Yamasaki}
 Y. Yamasaki.:
 Factorization formulas for higher depth determinants of the Laplacian on the $n$-sphere,
 preprint, 2010. arXiv:1011.3095

\end{thebibliography}
\end{document}